\documentclass[11pt,twoside]{article}
\usepackage{lipsum}
\usepackage{framed}
\usepackage[a4paper,left=2.5cm,right=2.5cm,top=2.5cm,bottom=2.5cm]{geometry}

\usepackage{xcolor, soul}
\sethlcolor{yellow}

\usepackage{algorithm}

\usepackage{mdframed}

\usepackage{parskip}
\usepackage{enumitem}
\setlist[itemize]{leftmargin=.25in}
\setlist[enumerate]{
  label=(\roman*),
  ref=\theassumption(\roman*),
  leftmargin=*
}
\AtBeginEnvironment{definition}{\vspace{-\parskip}\leavevmode}
\AtBeginEnvironment{proposition}{\vspace{-\parskip}\leavevmode}
\AtBeginEnvironment{theorem}{\vspace{-\parskip}\leavevmode}
\AtBeginEnvironment{assumption}{\vspace{-\parskip}\leavevmode}
\AtBeginEnvironment{lemma}{\vspace{-\parskip}\leavevmode}
\AtBeginEnvironment{corollary}{\vspace{-\parskip}\leavevmode}

\makeatletter
\AddToHook{cmd/appendix/before}{
  \def\cref@section@alias{appendix}
  \def\cref@subsection@alias{appendix}
}
\makeatother

\usepackage{amsmath, amsthm}
\usepackage{libertine}
\usepackage[libertine]{newtxmath}
\usepackage{microtype}

\definecolor{DarkRed}{RGB}{120, 0, 0}
\definecolor{DarkBlue}{RGB}{27, 69, 148}
\definecolor{DarkGray}{RGB}{60, 60, 60}
\usepackage{hyperref}
\hypersetup{
  colorlinks=true,
  linkcolor=DarkBlue,
  citecolor=DarkRed,
  urlcolor=DarkGray
}

\mdfdefinestyle{mdframedperso}{%
    linecolor=black,
    innertopmargin=2pt,
    innerbottommargin=8pt,
    skipabove=6pt,
    skipbelow=6pt,
        }
\mdfdefinestyle{mdframethm}{%
    linecolor=white,
    innertopmargin=2pt,
    innerbottommargin=8pt,
    skipabove=6pt,
    skipbelow=6pt,
    backgroundcolor=gray!15!white
        }

\usepackage{etoolbox}

\usepackage{graphicx}
\usepackage{stmaryrd}
\usepackage{epstopdf}
\usepackage{algorithmic}
\usepackage{bbm}
\usepackage[numbers]{natbib}
\usepackage{multirow}

\ifpdf
  \DeclareGraphicsExtensions{.eps,.pdf,.png,.jpg}
\else
  \DeclareGraphicsExtensions{.eps}
\fi

\newcommand{\R}{\mathbb{R}}

\newcommand{\N}{\mathbb{N}}

\newcommand{\dom}{\operatorname{dom}}
\newcommand{\Int}{\operatorname{int}}
\newcommand{\spnorm}[1]{%
  \left|\!\left|\!\left| #1 \right|\!\right|\!\right|}

\usepackage[nameinlink]{cleveref}
\theoremstyle{remark}
\newtheorem{remark}{Remark}[section]

\newtheorem{example}{Example}[section]

\theoremstyle{plain}

\newtheorem{assumption}{Assumption}

\crefname{hypothesis}{hypothesis}{Hypotheses}
\crefname{assumption}{assumption}{assumptions}
\Crefname{assumption}{Assumption}{Assumptions}

\author{
Maxence Adly$^{1}$,
Alix Chazottes$^{1,2,*}$,
Émilie Chouzenoux$^{1}$,
Jean-Christophe Pesquet$^{1}$ \&
Florent Sureau$^{2}$\\[1ex]
{\small $^{1}$ Université Paris-Saclay, CentraleSupélec, Inria, Centre pour la Vision Numérique, Gif-sur-Yvette, France}\\
{\small
$^{2}$ Université Paris-Saclay, Laboratoire BioMaps, CEA,
Orsay, France}\\
{\small
$^{*}$ Corresponding author: {\footnotesize \href{mailto:alix.chazottes@centralesupelec.fr}{alix.chazottes@centralesupelec.fr}}
}
}

\newlist{subassumptions}{enumerate}{1}
\setlist[subassumptions]{
  label=(\alph*), 
  ref=\theassumption(\alph*),
  parsep=0pt,
  leftmargin=*
}
\crefname{subassumptionsi}{assumption}{assumptions}
\Crefname{subassumptionsi}{Assumption}{Assumptions}
\crefname{subassumptions}{assumption}{assumptions}
\Crefname{subassumptions}{Assumption}{Assumptions}

\usepackage[caption=false]{subfig}
\usepackage[font=footnotesize, labelfont=bf]{caption}
\usepackage{stackengine} 
\usepackage{amsopn}

\usepackage{booktabs}

\newcommand{\dt}{\,\mathrm{d}t}

\usepackage{tikz}
\usetikzlibrary{arrows.meta, positioning, calc}

\newcommand{\majO}{\ensuremath{1}}
\newcommand{\majII}{\ensuremath{2}}
\newcommand{\majIII}{\ensuremath{3}}
\newcommand{\majVIII}{\ensuremath{4}}
\newcommand{\majI}{\ensuremath{5}}
\newcommand{\majVI}{\ensuremath{6}}
\newcommand{\majV}{\ensuremath{7}}
\newcommand{\majVp}{\ensuremath{8}}
\newcommand{\majVII}{\ensuremath{9}}

\makeatletter
\renewcommand{\maketitle}{%
  \begin{center}
    \rule{\textwidth}{0.4pt}\par
    \vspace{0.6em}
    {\LARGE \bfseries \@title \par}
    \vspace{0.6em}
    \rule{\textwidth}{0.4pt}\par
    \vspace{1.2em}
    {\normalsize \@author \par}
  \end{center}
  \vspace{1em}
}
\makeatother

\newmdtheoremenv[style=mdframethm]{theorem}{Theorem}[section]
\newmdtheoremenv[style=mdframethm]{proposition}{Proposition}[section]
\newmdtheoremenv[style=mdframethm]{lemma}{Lemma}[section]
\newmdtheoremenv[style=mdframethm]{corollary}{Corollary}[section]
\newmdtheoremenv[style=mdframethm]{definition}{Definition}[section]

\numberwithin{equation}{section}
\numberwithin{assumption}{section}

\newcommand{\reviewchanges}[1]{#1}

\usepackage[toc,page]{appendix}
\title{ Variable Bregman Majorization–Minimization algorithms \\for nonconvex nonsmooth optimization, \\ with application to Poisson imaging}

\newcommand{\shortauthors}{M. \textsc{Adly}, A. \textsc{Chazottes}, E. \textsc{Chouzenoux}, J.-C. \textsc{Pesquet} \& F. \textsc{Sureau}}
\newcommand{\shorttitle}{Variable Bregman MM for nonconvex nonsmooth optimization with applications to Poisson imaging}

\usepackage{fancyhdr}
\begin{document}
\thispagestyle{plain}
\maketitle

\begin{abstract} 
In this work, we introduce a unifying Bregman-based majorization-minimization (MM) framework for solving nonconvex nonsmooth optimization problems. The proposed approach leverages Bregman divergences, possibly varying across iterations, to construct tailored surrogate functions that majorize the objective.
We establish the convergence of the iterates of the resulting variable Bregman MM algorithm to critical points under the Kurdyka–\L{}ojasiewicz 
property, relaxing standard assumptions such as the Lipschitz smoothness of the nonconvex objective function. We derive a constructive methodology to build a broad class of variable Bregman majorants with tractable minimizers. Our study encapsulates various existing majorization techniques, in particular those derived for Poisson data fidelity terms in imaging inverse problems. Numerical experiments on Positron Emission Tomography (PET) image reconstruction with a nonconvex regularizer showcase the practical feasibility of the proposed scheme. 

\end{abstract}



\tableofcontents

\vspace{1cm}
\section{Introduction and related work}
\label{sec:intro}

Over the last decades, \emph{majorization–minimization} (MM) techniques \cite{Hunter2004,MM_surrogate,Sun2017} have emerged as a unifying approach for tackling a wide range of nonconvex optimization tasks \cite{Naderi2019, Hien2023,Kang2021,Repetti2021,Necoara2025}, with various successful applications in the field of imaging inverse problems~\cite{ChouzenouxPesquetRepetti2015,mmsubspaceL2L0,Zhou2024,Li_2025,Chouzenoux2023}. 
By iteratively constructing convex surrogates (or majorants) of the nonconvex objective, these methods provide a stable descent mechanism at each iteration. Furthermore, as shown for instance in  \cite{mmsubspaceL2L0,Hien2023}, MM techniques can benefit from proven convergence of their iterates toward a critical point in the challenging nonconvex setting, by leveraging powerful proof mechanisms based on Kurdyka–\L{}ojasiewicz (K\L{})  assumption introduced in~\cite{Attouch2011}. Convergence results highly depend on the structure of the majorant functions, and most existing theoretical results on MM methods assume majorization inequalities that hold within a Euclidean geometry (i.e., quadratic majorant functions), with an underlying metric possibly varying along iterations \cite{ChouzenouxPesquetRepetti2015,Repetti2021}. 

However, such a Euclidean structure does not fit well classes of problems where the objective function is not Lipschitz differentiable on its domain. A typical example,  widely encountered in computational imaging, arises with functions containing (minus) logarithmic terms, possibly composed with linear operators. The barrier behavior of the logarithm around zero makes the objective function and its gradient unbounded at the border of its domain, preventing the construction of meaningful quadratic majorant functions. Various MM techniques have been proposed in the past to account for this situation, motivated by the applicative problem of image restoration/reconstruction problems under Poisson noise \cite{Bertero_2009}, starting from the historical maximum likelihood expectation maximization (ML-EM) algorithm \cite{Dempster1977,DePierro1993} and its multiple extensions \cite{Fessler1993,ErdoganFessler1999}. Convergence has been studied for instance in \cite{EGGERMONT199025,Bertero2022,Jacobson2007}, leveraging properties of the Kullback-Leibler divergence and MM update rules. However, there remains a crucial need to unify these methods and to better characterize their convergence guarantees in the general nonconvex nonsmooth case, in order to address a wider class of penalized optimization problems.

A key concept adopted in this work is the use of \emph{Bregman divergences} whose fundamental role in optimization has been pointed out in main contributions
\cite{BauschkeBorwein1997,BauschkeCombettes2017,Mukkamala2021,Nguyen2016},
rooted in the literature of variational analysis \cite{Rockafellar1970,RockafellarWets}. Bregman majorant functions varying along iterations, and the consequent variable Bregman MM algorithms, were introduced in our preliminary study \cite{Rossignol2022}. Bregman divergences provide a more general framework than Euclidean metrics, including in particular the Kullback-Leibler divergence. They hence lead to more flexible choice of majorant approximations, beyond classical quadratic bounds, that not only better fit singular functions such as logarithmic barriers, but naturally account for positivity constraints arising in image processing. The recent paper \cite{segolene2025_vbmm_dirichlet} proves the subsequential convergence of the variable Bregman MM method in the convex case, and displays its application to maximum likelihood estimation. 

\reviewchanges{The contributions of this paper are as follows:
\begin{itemize}
    \item We prove the convergence of the variable Bregman MM iterates in the nonconvex, nonsmooth setting under the K\L{} assumption. To do so, we establish the well-posedness of our algorithm under our technical assumptions and demonstrate sufficient decrease and error conditions, showing that the proposed scheme fits into the \emph{gradient-like descent} framework from~\cite{BolteFirstOrder}.
    \item We introduce a constructive and unifying analysis for the design of variable Bregman majorant functions, placing particular emphasis on logarithmic objectives, an important use case widely explored in the computational imaging literature. Thus, our framework encompasses several ad hoc EM-based algorithms used as state-of-the-art solvers for Poisson imaging inverse problems and enables the construction of new algorithms. 
    \item We perform numerical benchmarking on a tomographic image reconstruction problem under Poisson noise with nonconvex regularization. Through an ablation study, we illustrate the benefits of variable non-Euclidean majorant metrics, both in terms of image quality and convergence speed.
\end{itemize}}


Related concepts can be found in the works of~\cite{Hurault2023,bauschke_descent_2017,Mukkamala2021}, which also rely on Bregman algorithmic updates with convergence guarantees in a nonconvex setting. However, these works did not consider adaptive Bregman majorant functions, in the sense that the induced Bregman metric was not allowed to vary along iterations. In contrast, our approach considers variable Bregman metrics, which naturally extends variable metric Euclidean-based algorithms such as~ \cite{Frankel_Garrigos_Peypouquet,ChouzenouxPesquetRepetti2015,bonettini2016}. \reviewchanges{As shown in our numerical experiments on Poisson image reconstruction, combining \emph{variable} metrics with \emph{non-quadratic} majorants is essential for achieving both fast convergence in practice and high image quality.}  

The rest of this paper is structured as follows.
First, in \cref{sec:preliminaries} we introduce the core mathematical tools such as Bregman divergences and the K\L{} inequality. 
In \cref{sec:general-pb} we present the optimization problem, the main assumptions, and our proposed algorithm.
\Cref{sec:convergence-analysis} presents convergence results for our Bregman MM algorithm. 
In \cref{sec:unifying-framework} we develop a unifying framework for constructing Bregman tangent majorants and propose a wide family of such majorants for Poisson data fidelity terms. Finally, in \cref{sec:Appli-PET} we illustrate the applicability of the proposed algorithm and discuss the comparative performance of the majorants in the context of Positron Emission Tomography (PET) image reconstruction.

\section{Preliminaries}
\label{sec:preliminaries}
In this section, we provide a set of optimization tools that will be useful in deriving our algorithm and analyzing its convergence.

\textbf{Notation:} We denote by $\Gamma_0(\R^N)$ the set of proper lower semicontinuous convex functions $f: \R^N \to \R \cup \{+\infty\}$. The domain of $f$, i.e., the set of vectors $x$ in $\R^N$, such that $f(x) < +\infty$, is denoted by $\dom(f)$. The graph of $f$, i.e., the set of pairs $(x,f(x))$, for every $x \in \R^N$, is denoted by $\operatorname{graph}(f)$. For a subset $A \subset \R^N$, we denote $\bar{A}$ its closure, $\Int(A)$ its interior, and $\operatorname{bd}(A) = \bar{A}\setminus\Int(A)$ its boundary. $\mathbf{I}_N$ is the identity matrix of $\mathbb{R}^N$.
The inner product is denoted by $\langle \cdot,\cdot \rangle$
and the Euclidean norm is $\|\cdot\|$.

    \subsection{The Bregman framework} 
   
    We begin by providing useful reminders on the framework of Bregman divergences induced by Legendre functions, as defined for example in \cite{Rockafellar1970},\cite{BauschkeBorwein1997}.

    \begin{definition}{\textbf{(Legendre function)}}\\
    Let $h\in\Gamma_0(\R^N)$. Function $h$ is called  \begin{itemize}
        \item[$\bullet$] \textsc{Essentially smooth} if $h$ is differentiable on $\Int(\dom(h))$ and, for all sequences $(x^k)_{k \in \N} \in \R^N$, converging to some $x^\ast \in \operatorname{bd}(\dom(h))$, we have $\Vert \nabla h(x^k) \Vert \xrightarrow{k \to \infty} +\infty$;   
        \item[$\bullet$] Of \textsc{Legendre type} if $h$ is essentially smooth and strictly convex on $\Int(\dom(h))$.
        
    \end{itemize}
    We denote by $\mathcal{L}(\R^N)$ the class of Legendre functions on $\R^N$.
    \end{definition}

    The class of Legendre functions benefits from properties that will be useful for the analysis. In particular, let us remind the gradient bijection property.

    \begin{proposition}{\textbf{(Gradient of a Legendre function)}} \cite{Rockafellar1970} \label{prop:gradlegendre}\\
        Let $h \in \mathcal{L}(\R^N)$, we denote by $h^\ast$ its Legendre-Fenchel conjugate. Then $\nabla h$ is a bijection from $\Int(\dom(h))$ to $\Int(\dom(h^\ast))$ and its inverse is the gradient of the conjugate: $(\nabla h)^{-1} = \nabla h^\ast$.
        
    \end{proposition}

    Given the class of Legendre functions, the Bregman distance can be defined as follows.
    
    \begin{definition}{\textbf{(Bregman Distance)}}\\
        Let $h \in \mathcal{L}(\R^N)$. The \textsc{Bregman distance} associated with $h$ evaluated at $(x,y)\in \dom(h)\times\Int(\dom(h))$ reads 
        \begin{equation}
            D_{h}(x,y) = h(x)-h(y) - \langle \nabla h (y), x-y\rangle.\label{eq:Bregman}
        \end{equation}
    \end{definition}
The above definition basically expresses the Bregman distance, as the difference at $x$ between $h$ at $x$ and its first order Taylor expansion around $y$.


    \begin{example}
            Let us now give a few important examples of Bregman distances that are widely used in the optimization literature, as well as their associated domain of validity. 
        \begin{enumerate}
            \item[$\bullet$] \textsc{Euclidean distance:} The Bregman distance generated by $h: x \longmapsto \| x \|^2$ is equivalent to the regular Euclidean distance. 
            
            \item[$\bullet$] \textsc{Itakura–Saito distance:} The Bregman distance generated by $h: x = (x_n)_{1\le n \le N} \in (0,+\infty)^N \longmapsto -\sum_{n=1}^N \log(x_n)$ \textit{(Burg's entropy)} is 
            sometimes referred to, as the Itakura\-–Sai\-to distance (or divergence) \cite{Banerjee2004, Fevotte2009}.
             
            \item[$\bullet$] \textsc{Kullback-Leibler divergence:} \reviewchanges{The Bregman distance generated by $h: x= (x_n)_{1\le n \le N} \in [0,+\infty)^N \longmapsto \sum_{n=1}^N x_n\log(x_n)$ \textit{(the negative Shannon's entropy)} is the generalized Kullback--Leibler divergence.} 

            The above function will play a fundamental role in our analysis in \cref{sec:unifying-framework}, and in our experiments in \cref{sec:Appli-PET}. 
        \end{enumerate}
    \end{example}

    \noindent These examples illustrate the flexibility of Bregman divergences, which can be tailored to different data structures and problems.



    
    \subsection{Variational analysis}
    We now present the subdifferentiation tools necessary for our convergence proof.
    The classical definition of the limiting subdifferential $\partial f$ of a proper lower-semicontinuous function
    $f: \R^N \to \R\cup\{+\infty\}$ will be used.
        
%
%
%
A fundamental property of the limiting subdifferential is its \textit{closedness}. 
For any sequence $(x^k,u^k)_{k\in \N}\in \text{graph}(\partial f)^\N$ converging to $(\bar{x},\bar{u})$, if $(f(x^k))_{k\in \N}$ converges to $f(\bar{x})$, then $$(\bar{x},\bar{u}) \in \text{graph}(\partial f).$$
The above result is useful to prove that the limit point of a generated sequence is a critical point, that is an element of the set 
defined as
\begin{equation}
    \text{crit}(f)=\{x \in \R^N\mid 0 \in \partial f(x)\}.
\end{equation}




    \subsection{Kurdyka-\L{}ojasiewicz property}
    
    Finally, we recall here an essential property that will be key in the derivation of the global convergence of the algorithm.
    
    \begin{definition}
    \cite{Bolte2014}
\begin{itemize}
\item For every $\eta>0$, let 
    \begin{equation}
        \Phi_\eta = \Big\{\phi \in \mathcal{C}^0([0,\eta))\cap\mathcal{C}^1((0,\eta))\mid \phi(0)=0, \; \phi \text{ concave },\; \phi'>0 \text{ on } [0,\eta) \Big\}
    \end{equation}
be the class of \textsc{desingularizing} functions
on $[0,\eta)$.
\item 
        A function $f: \R^N \to \R\cup\{+\infty\}$ is said to satisfy the \textsc{Kurdyka–\L{}ojasiewicz} (K\L{}) property at $\Bar{x} \in \dom(\partial f)$ if there exist $\eta > 0$,  a neighborhood $U$ of $\Bar{x}$,       
and $\phi\in \Phi_\eta$
%
    such that, for every $x \in U \cap [f(\Bar{x})<f(x)<f(\Bar{x})+\eta]$, the \textsc{K\L{} inequality} hereafter holds: 
    \vspace{-0.5cm}
    \begin{equation}
        \phi'(f(x)-f(\Bar{x}))d(0,\partial f(x)) \geqslant 1,
    \end{equation}
 where \reviewchanges{$d(0,\partial f(x)):= \inf \{\|u\|\mid u \in \partial f(x)\}$ denotes the distance from the origin to the limiting subdifferential $\partial f(x)$}.
    \end{itemize}
    \end{definition}

    One can see that the K\L{} property is trivially verified at non critical points of a lower-semicontinuous proper function $f$ (by choosing $\phi$ linear).

    \begin{theorem}{\textbf{(Uniformized K\L{} property)}}\cite[Lemma 3.6]{Bolte2014}\\
    \label{thm:uniform-KL}
    Let $K$ be a nonempty compact subset of $\R^N$ and let $f: \R^N \to \R\cup\{+\infty\}$ be a proper lower-semicontinuous function. Assume that $f$ is constant on $K$ and that it satisfies the K\L{} property at each point of $K$. Then there exist $\eta >0$, $\varepsilon >0$, and $\phi \in \Phi_\eta$ such that 
    \begin{align}
        (\forall \Bar{x}\in K) (\forall x \in \R^N) 
     \begin{cases}
     d(x,K)<\varepsilon\\
     f(\Bar{x})<f(x)<f(\bar{x})+\eta
     \end{cases}
     \; \Longrightarrow\;
    \phi'(f(x)-f(\Bar{x}))
    d(0,\partial f(x))
    \geqslant 1.
    \end{align}
    \end{theorem}


\section{Mathematical model} \label{sec:general-pb}

We are now ready to present the problem formulation, our assumptions, and the proposed algorithm.

    \subsection{Problem formulation}
    In what follows, we will be interested in the resolution of the following non necessarily convex optimization problem:
    \begin{equation}
    \label{eq:mainpb}
          \quad \operatorname*{minimize}_{x \in \mathbb{R}^N} F(x) = f(x)+g(x).
    \end{equation}
   We start by stating below our first set of assumptions on the involved functions in \eqref{eq:mainpb}.
   
    \begin{mdframed}[style=mdframedperso]
    \begin{assumption}\label{ass:A1}
    In problem \eqref{eq:mainpb},
    \vskip 3mm
        \begin{subassumptions}
        
            \item\label{ass:A1:i} 
\reviewchanges{$f:\mathbb{R}^N\to\mathbb{R}\cup\{+\infty\}$ is a proper function such that
$\mathcal{D}:=\operatorname{int}(\operatorname{dom} f)$ is nonempty and
$f$ is continuously differentiable on $\mathcal{D}$;}
             \vskip 3mm
             \item\label{ass:A1:ii} \reviewchanges{$g: \R^N \to \R\cup\{+\infty\}$} is a proper lower semi-continuous convex function with $\dom g \subset \mathcal{D}$;  

            \vskip 3mm
            \item\label{ass:A1:iv} Function $F$ is coercive, continuous on its domain and lower bounded by $\underline{F} > -\infty$.
        \end{subassumptions}
    \end{assumption}
    \end{mdframed}

    \Cref{ass:A1:i,ass:A1:ii} 
 are standard when solving composite continuous optimization problems with proximal gradient approaches.  \Cref{ass:A1:iv} is a typical hypothesis ensuring the well posedness of the minimization problem, and, in particular, the existence of solutions.

    \begin{definition}{\textbf{(Bregman Tangent Majorant)}}\\
        Let $f$ be a differentiable function on an open set $\mathcal{D}\subset \R^N$, $z \in \mathcal{D}$ a reference point and $h_z \in \mathcal{L}(\R^N)$. Let us define
        \begin{equation}
           (\forall x \in \mathcal{D}) \quad  Q_f(x,z) = f(z) + \langle \nabla f(z), x-z \rangle + D_{h_z}(x,z). \label{eq:Bregmandef1}
        \end{equation}
        $Q_f(\cdot,z): \mathcal{D} \to \R$ in \eqref{eq:Bregmandef1} is a \textsc{Bregman Tangent Majorant} of $f$ at $z \in \mathcal{D}$, associated with $h_z$, if the following majorizing inequality holds:
        \begin{equation}
        (\forall x \in \mathcal{D}) \quad 
            f(x) \leqslant Q_f(x,z). \label{eq:Bregmandef2} 
        \end{equation}
\label{def:BregmanMaj}
\vspace{-0.5cm}
    \end{definition}
    The key original element of our construction is that the Bregman metric function $h_z$ can vary according to the reference point $z$, allowing the majorant to dynamically adjust, thereby encapsulating more complex local geometrical constraints of the function we wish to majorize. This is in relation with the previous work in \cite{ChouzenouxPesquetRepetti2015}, also based on variable metrics, although it was restricted to the Euclidean case. See \cref{other-algs} for more discussion.

    To guarantee the applicability of our framework, we introduce the following assumption. 
    
    \begin{mdframed}[style=mdframedperso]
    \begin{assumption}
    \label[assumption]{ass:A2}
    Function $f$ in problem \eqref{eq:mainpb}, admits a Bregman tangent majorant at every point $z \in \mathcal{D}$, associated with a family of {Legendre} functions $(h_z)_{z \in \mathcal{D}}$ satisfying $$(\forall z \in \mathcal{D}) \quad \mathcal{D} \subset \Int(\dom(h_z)).$$
    \end{assumption} 
    \end{mdframed}

\Cref{sec:unifying-framework} will provide examples of constructions and properties of Bregman tangent majorants, for an important class of functions in image recovery, involving a Kullback-Leibler divergence composed with a linear operator.
    


    \subsection{Proposed algorithm}

    Under Assumptions \ref{ass:A1} and \ref{ass:A2}, we can construct our proposed MM algorithm to solve problem \eqref{eq:mainpb}. Starting from an element in $\dom g \subset \mathcal{D}$, we iteratively define the next iterate as the minimizer of a Bregman majorant approximation to $F = f + g$ at the current point. To do so, we rely on the following inequality, direct consequence of \eqref{eq:Bregmandef2},  for some $z \in \dom g$:
\begin{equation}
    (\forall x \in \dom g) \quad F(x) \leq Q_f(x,z)+g(x).
\end{equation}
    The resulting numerical solution is presented in  \cref{alg:VBMM}. 
        
        
            
        \begin{algorithm}[H]
            \caption{Variable Bregman MM Algorithm to solve problem \eqref{eq:mainpb}}
            \label{alg:VBMM}
            \begin{algorithmic}
                \STATE \textbf{Input:}
                \STATE \quad Initial point $x^0 \in \dom{g}$
            
                \FOR{$k = 0, 1,  \ldots$}
                    \STATE $x^{k+1} = \operatorname*{argmin}_{x \in \R^N} \left(Q_f(x, x^k) + g(x)\right)$
                \ENDFOR
            
            \end{algorithmic}
        \end{algorithm}

        \begin{remark}
            For generic choices of $Q_f$ and $g$, the optimization problem to solve at each iteration may remain challenging. In \cref{sec:unifying-framework,sec:Appli-PET}, we will present specific construction choices for which the corresponding updates become simple and easy to compute.
        \end{remark}
        
        \begin{remark}\label{remarkEquivalent}
        Let us then provide equivalent forms of the $k$-th iteration of \cref{alg:VBMM} that will be useful in our analysis, using \eqref{eq:Bregmandef1} and \eqref{eq:Bregman}:
    \begin{align} 
    x^{k+1} & = \operatorname*{argmin}_{x \in \R^N} \left(Q_f(x, x^k) + g(x)\right)\\ \nonumber
    & = \operatorname*{argmin}_{x \in \dom g} \left(f(x^k) + \langle \nabla f(x^k), x-x^k \rangle + D_{h_{x^k}}(x,x^k)  + g(x) \right)\\ \nonumber
     & = \operatorname*{argmin}_{x \in \dom g} \left(\langle \nabla f(x^k) - \nabla h_{x^k}(x^k),x\rangle + h_{x^k}(x)+g(x)\right). 
        \end{align}
        \end{remark}
    
We now establish the well definedness of the proposed algorithm.
    \begin{proposition}
    \label{prop:welldef}
    {\textbf{(Well-definedness)}}\\ 
            Under Assumptions \ref{ass:A1} and \ref{ass:A2}, \cref{alg:VBMM} is well defined and generates a sequence $(x^k)_{k \geq 1}$ in $\dom g \subset \mathcal{D}$. Moreover, the iteration satisfies 
             \begin{equation}
                (\forall k \in \N) \; \left(\exists v^{k+1} \in \partial g(x^{k+1})\right) \quad
                x^{k+1} = \nabla h_{x^k}^\ast(\nabla h_{x^k}(x^k)-\nabla f(x^k)-v^{k+1}). \label{eq:iteration-generale}
            \end{equation}
    \end{proposition}

        \begin{proof}
        We proceed by induction. The initial point $x^0 \in \dom{g}$. Let $k \in \N$ be such that $x^k \in \mathcal{D}$. 
    The function      
            $ \gamma_{x^k}: x \longmapsto\langle \nabla f(x^k) - \nabla h_{x^k}(x^k),x\rangle + h_{x^k}(x)+g(x)
            $, which appears in
            the final expression in Remark \ref{remarkEquivalent}, is well defined as $\dom g \subset \mathcal{D} \subset \Int(\dom h_{x^k} )$. Moreover, it is proper and lower semi-continuous (lsc) by \cref{ass:A1:ii}. Finally, up to an additive constant independent on $x$, $\gamma_{x^k}$ majorizes $F$, which is coercive by \cref{ass:A1:iv}. Hence  $\gamma_{x^k}$ is coercive, and admits a minimizer belonging to $\dom g  \subset\mathcal{D}$. Moreover, since $h_{x^k}$ is of Legendre type, it is strictly convex on $\Int(\dom h_{x^k} )$ which implies that $\gamma_{x^k}$ is strictly convex on $\Int(\dom h_{x^k} ) \supset \mathcal{D} \supset \dom g$, ensuring the uniqueness of its minimizer on $\dom g$. This minimizer $x^{k+1}\in \dom g$ verifies the optimality condition $0 \in \partial \gamma_{x^k}(x^{k+1})$ where
            \begin{equation}
                \partial \gamma_{x^k}(x^{k+1})
            \,=\,
            \nabla f(x^k) - \nabla h_{x^k}(x^k)
            \;+\;
            \nabla h_{x^k}(x^{k+1})
            \;+\;
            \partial g(x^{k+1}).
            \end{equation}
            Hence,
            \begin{equation}
                (\exists v^{k+1} \in \partial g(x^{k+1})) \quad \nabla h_{x^k}(x^{k+1}) = \nabla h_{x^k}(x^k)-\nabla f(x^k) - v^{k+1}.
            \end{equation}
            Using \cref{prop:gradlegendre} allows us to express $x^{k+1}$ as in 
            \eqref{eq:iteration-generale}.
        \end{proof}

    \subsection{Relation with existing algorithms}\label{other-algs} To illustrate the versatility of \cref{alg:VBMM}, in this section we show how it encompasses several other well‑known methods.
     \subsubsection{Variable metric forward-backward algorithm}

     Let us first consider the case when \cref{ass:A2} holds with the following Legendre generator:
        \begin{equation}
            (\forall x \in \R^N)(\forall z \in \mathcal{D}) \quad h_{z}(x) = \frac12 x^\top M(z) x
        \end{equation}
        where $M(z) \in \R^{N \times N}$ is symmetric positive definite.
        The Bregman distance associated with this generator is given  by the squared norm in the metric induced by $M(z)$: 
        \begin{equation}
        (\forall x \in \R^N)(\forall z \in \mathcal{D}) \quad D_{h_{z}}(x,z) = \frac12 (x-z)^\top M(z) (x-z) = \frac12 \Vert x-z\Vert^2_{M(z)}.
        \end{equation}


        Plugging the corresponding Bregman tangent majorant
        into \cref{alg:VBMM} and defining $M_k:=M(x_k)$ yields the following iteration, known as the \textit{Variable Metric Proximal Gradient algorithm} (VMFB) \cite{ChouzenouxPesquetRepetti2015}: 
\begin{equation*}
            x^{k+1} =\arg\min_x \left\{\left\langle \nabla f(x^k),\, x- x^k\right\rangle + \frac12 \Vert x - x^k\Vert^2_{M_k} + g(x)\right\} = \text{prox}_g^{M_k}\left(x^k-M_k^{-1}\nabla f(x^k)\right).
        \end{equation*}
In particular, if $f$ is $L_f$-Lipschitz smooth, one can set $M_k \equiv \reviewchanges{\alpha_k^{-1}} \mathbf{I}_N$, with $\alpha_k \reviewchanges{\leq \frac{1}{L_f}}$, and recover the classic \textit{forward-backward} (a.k.a. proximal gradient) algorithm.
%
%
    The convergence of the above algorithms has been extensively studied \cite{bonnans1995,ChouzenouxPesquetRepetti2015, combettes2005, Combettes01092014, bonettini2016}, in convex and non-convex scenarios.

    \subsubsection{Bregman Proximal Gradient algorithm}
        In practice, the above Euclidean majorant construction might not exist when the gradient Lipschitz continuity condition on $f$ is not met. The Bregman minimization methods \cite{Mukkamala2021, Hurault2023} allow this issue to be addressed  by relying on \textit{Lipschitz-like convexity condition}. For a given Legendre function $h$, the pair $(f,h)$ satisfies this condition if there exists $L>0$ 
        such that 
        \begin{equation}\label{eq:LC}
            Lh-f \text{ is convex}.
        \end{equation}
        Under this condition, the extended descent lemma in \cite{bauschke_descent_2017} shows that 
        $$\eqref{eq:LC}\Longleftrightarrow (\forall (x,y) \in \Int(\dom h)^2) \quad f(x) \leqslant f(y) + \langle \nabla f(y), x-y\rangle + LD_h(x,y)=Q_f(x,y).$$
 %
        Using the above Bregman tangent majorant,      
         \cref{alg:VBMM} reduces to the Bregman Proximal Gradient (BPG) algorithm. 
        In practice, \reviewchanges{the BPG iteration is used with a stepsize sequence $\eta_k$ that satisfies, for every $k \in \N$, $0 < \eta_{\min} \leq \eta_k \leq \eta_{\max} < 1/L$}.
        The BPG update with stepsize is given by 
        \begin{equation} (\forall k \in \N) \quad
            x^{k+1} = \operatorname*{argmin}_{x \in \R^N} \Big\{\langle \nabla f(x^k), x-x^k\rangle + \frac{1}{\eta_k}D_h(x,x^k) + g(x)\Big\}.
        \end{equation}      
By setting $h_{x^k} = \frac{1}{\eta_k} h$ in \cref{alg:VBMM}, one recovers the BPG algorithm. This shows that this form of the BPG emerges as a special case of VBMM where all the generating functions $(h_{x^k})_{k \in \N}$ are proportional to each other. Our proposed \cref{alg:VBMM} can thus be viewed as a generalization of the \emph{BPG Algorithm}. The difference is that in BPG, the control over the stepsize from $x^k$ to $x^{k+1}$ is given by the scalar parameter $\eta_k$ whilst in our scheme, it is given by the choice of our point-dependent generator $h_{x^k}$. This adaptability will be illustrated in our experimental scenario, coping with image reconstruction under Poisson statistics.

Remark that, in the study of \cite{bauschke_descent_2017}, the stepsize range for BPG convergence is extended, by exploiting the symmetry of the Legendre function $h$, leading to the condition $\eta_k \in (0,((1+\alpha)-\delta)/L)$, with $\delta \in (0,1+\alpha)$ and $\alpha$ the symmetry coefficient of $h$. As emphasized in \cite[Rem. 3.2]{segolene2025_vbmm_dirichlet}, such exploitation of the symmetry properties can also be made in the case of varying Legendre metrics. In the present paper, for ease of presentation, we have opted for a versatile convergence analysis, not depending on symmetry properties of the Legendre functions  $(h_{x^k})_{k \in \N}$.

        
%
%

        

    \subsubsection{Variable Metric Mirror Descent algorithm}
    
    Let us now consider problem \eqref{eq:mainpb} with $g = 0$, and choose a proportional family of Legendre generators as in BPG, i.e., $(\forall k \in \N) \quad h_{x^k} = \frac{1}{\eta_k}h$. Then,
    \cref{alg:VBMM} reads as an instance of the \textit{Mirror Descent} (MD) algorithm~\cite{BeckTeboulle2003}:
    \begin{equation} \label{eq:mirror-desc}
        (\forall k \in \N) \quad x^{k+1} = \operatorname*{argmin}_{x \in \R^N}\Big\{\langle \nabla f(x^k), x-x^k\rangle + \frac{1}{\eta_k}D_{h}(x,x^k)\Big\}.
    \end{equation}

    In \cite{Teboulle2018}, the authors show the convergence of MD, by relaxing the standard Lipschitz smoothness assumption into \eqref{eq:LC}, and by assuming, in addition, the following inequality:
    \begin{equation}
        (\forall \eta > 0)\; (\exists G >0)\; \bigl(\forall x \in \Int(\dom(h))\bigr) \quad \langle\nabla f(x^k),x-x^k\rangle - \eta D_{h}(x,x^k) \leqslant \frac{\eta}{2}G^2.
    \end{equation} 

Our \cref{alg:VBMM} with $g \equiv 0$ can be understood as a variable metric extension of MD, where the Legendre functions can vary along iterations in a very flexible manner (i.e., non necessarily based on a proportionality rule). Convergence of this scheme under convexity assumptions was studied in \cite{segolene2025_vbmm_dirichlet} applied to Dirichlet MLE. In this work, we extend these convergence results to non-convex cases, under Lipschitz smoothness assumption.

    \section{Convergence Analysis} 
    \label{sec:convergence-analysis}
    
    In this section we establish the convergence properties of the proposed Variable Bregman MM algorithm, stated in \cref{alg:VBMM}. We will show that under suitable assumptions, the sequence  $(x^k)_{k \in \N}$ generated by \cref{alg:VBMM} forms a \textit{gradient-like descent sequence}  for minimizing $F$ in the sense of \cite[Definition 6.1]{BolteFirstOrder}. As shown in \cite[Theorem 6.2]{BolteFirstOrder}, this property coupled with a K\L{} assumption on function $F$ in problem \eqref{eq:mainpb}, yields the desired global convergence result.
    
    Let $x^0 \in \dom{g}$ be an initial point. In what follows, we will denote by $\Omega(x^0)$ the set of cluster points of $(x^k)_{k \in \N}$ starting from point $x^0$. Note that by virtue
    of \cref{prop:welldef}, for all  $k \geq 1$ we have $x^k \in \dom{g}$.
    
    \begin{proposition}{\textbf{(Objective convergence)}}\\
    \label{pro:obj-cv}
    Suppose that Assumptions \ref{ass:A1} and \ref{ass:A2} hold true. Then the sequence $(F(x^k))_{k \in \N}$ produced by \cref{alg:VBMM} is monotonically  decreasing, and converges. \reviewchanges{Moreover, $(x^k)_{k\in\mathbb N}$ is contained in the compact sublevel set
$
\{x\in\operatorname{dom}F\mid
F(x)\le F(x^0)\}$,
and therefore admits a convergent subsequence converging to some
$\bar x\in\Omega(x^0)$.}
    \end{proposition}

    \begin{proof}
        For all $k \in \N$, by Proposition \ref{prop:welldef}, $x^{k+1} = \operatorname*{argmin}_{x \in \R^N} Q_f(x,x^k)+g(x)$, therefore 
        \begin{equation}
            Q_f(x^{k+1},x^k)+g(x^{k+1}) \leqslant Q_f(x^k,x^k) + g(x^k) = F(x^k),
        \end{equation}
        where the last equality follows from the tangency property of $Q_f$.
        Moreover, since $Q_f(\cdot,x^k)$ majorizes $f$, 
        \begin{equation}
          F(x^{k+1}) = f(x^{k+1})+g(x^{k+1}) \leqslant Q_f(x^{k+1},x^k)+g(x^{k+1}).  
        \end{equation}
        Combining the above inequalities, it follows that $F(x^{k+1}) \leqslant F(x^k)$, for all $k \in \N$. Hence the sequence $(F(x^k))_{k \in \N}$ is non increasing, bounded from below by $\underline{F}$ (by \cref{ass:A1:iv}), therefore it converges. Furthermore, for every $k\in\mathbb{N}$,
\[
x^k\in
\operatorname{lev}_{\leq F(x^0)} F
:=
\bigl\{
x\in\operatorname{dom}F
\;|\;
F(x)\leq F(x^0)
\bigr\}.
\]
By \cref{ass:A1:iv}, $F$ is lower semicontinuous and coercive.
Consequently, this sublevel set is closed and bounded, hence compact
by the Heine--Borel theorem. Therefore, $(x^k)_{k\in\mathbb{N}}$
admits a convergent subsequence by the Bolzano--Weierstrass theorem.
Its limit belongs to $\Omega(x^0)$ by definition.  
        
    \end{proof}

    We now make the following assumptions on the smooth part of the cost function, i.e. $f$, and on its majorants to ensure that the generated sequence $(x^k)_{k \in \N}$ satisfies the conditions to be a gradient-like descent sequence for solving problem \eqref{eq:mainpb}.

    \begin{mdframed}[style=mdframedperso]
    \begin{assumption}
    \label{ass:A3}
    Consider \cref{alg:VBMM} applied to problem \eqref{eq:mainpb} under Assumptions \ref{ass:A1} and \ref{ass:A2}.
    \vskip 3mm
    \begin{subassumptions}
            \item\label{ass:A3:i} There exists $\underline{\gamma}>0$ such that, for every $k \in \N$, for every $x \in \dom{g}$, $Q_f(x, x^k) - f(x) \geq \frac{ \gamma_k}{2}\|x-x^k \|^2$, with $\gamma_k \in [\underline{\gamma},+\infty)$.

        \vskip 3mm
        \item\label{ass:A3:ii} There exists $\overline{L}>0$ such that, for every $k \in \N$, {$\nabla f- \nabla h_{x^k}$} is $L_k$-Lipschitz on every compact subset of $\dom{g}$, with $L_k \leq \overline{L}$.

    \end{subassumptions}
    \end{assumption}
    \end{mdframed}
    \vskip 3mm

     We now provide some examples of iterative algorithms in which Assumption \ref{ass:A3} is satisfied. 

    \begin{example}
    \label{exampleAss42}
    \begin{enumerate}[label=(\roman*)]
    
    

    \item Let us consider the VMFB algorithm \cite{ChouzenouxPesquetRepetti2015}, where the update relies on the majorant function given by
    \begin{equation}
   (\forall k \in \mathbb{N})(\forall x \in \mathbb{R}^N) \quad Q_f(x,x^k) = f(x^k) + \langle \nabla f(x^k),x-x^k\rangle + \frac{1}{2} (x-x^k)^\top M_k (x-x^k),
    \end{equation}
    with $M_k$ symmetric positive definite. For $\nabla f$ $L_f$-Lipschitz, direct calculations, combined with the descent lemma on $f$, show that, 
    \begin{align*}
        (\forall k \in \mathbb{N})(\forall x \in \dom{g}) \quad Q_f(x,x^k) - f(x)\geqslant \frac12
         (x-x^k)^\top (M_k-L_f \mathbf{I}_N) (x-x^k),
    \end{align*}
    which shows that \cref{ass:A3:i} is satisfied as soon as
\begin{equation}
    (\forall k \in \N) \quad M_k \succ L_f \mathbf{I}_N, \label{eq:exampleAss42}
\end{equation}
and $\underline{\gamma} = \min_k (\eta_{\rm min}(M_k) - L_f)>0$. Moreover, \cref{ass:A3:ii} holds if $\overline{L} = \max_k ({L_f}+\eta_{\rm max}(M_k))< +\infty$.
The standard forward-backward algorithm arises as a special case when, for every $k\in \N$, $ M_k = (\alpha_k)^{-1} \mathbf{I}_N$ where $\alpha_k$ is the algorithm stepsize assumed to be bounded.

    \item In the BPG algorithm, the majorant allowing us to build the iterates is given by 
    \begin{equation}
         (\forall k \in \mathbb{N})(\forall x \in \mathbb{R}^N) \quad Q_f(x,x^k) = f(x^k) + \langle \nabla f(x^k), x-x^k \rangle + \frac{1}{\eta_k}D_h(x,x^k).
    \end{equation}
    Under the  Lipschitz-like  condition \eqref{eq:LC} with constant $L>0$, using the generalized descent lemma,
    \begin{equation}
        Q_f(x,x^k) - f(x) \geqslant \Big(\frac{1}{\eta_k} - L\Big)D_h(x,x^k).
    \end{equation}

    Hence in this case, \cref{ass:A3:i} holds if $h$ is strongly convex and {$1/\eta_k - L> \epsilon$}, 
    with some constant $\epsilon>0$.  
\cref{ass:A3:ii} holds 
if, for every $k\in \N$, $\nabla f -\eta_k^{-1} \nabla h$ is Lipschitz on every compact of $\dom g$, with a Lipschitz constant independent from $k$. In particular, this is satisfied if both $\nabla f$
and $\nabla h$ are Lipschitz continuous on $\dom g$, and $\eta_k$ is bounded.


    \end{enumerate}
    \end{example}

    Other examples are discussed in \cref{sec:unifying-framework}, presenting a constructive framework for majorants tailored for Poisson data fidelity terms. 
    
    \subsection{Sufficient decrease condition} We derive here the first required condition for $(x^k)_{k \in \mathbb{N}}$  to be a \textit{gradient-like descent sequence} for minimizing $F$.
    \begin{proposition}{\textbf{(Sufficient decrease condition)}}\\
    \label{pro:sufficient-decrease}
   Suppose that Assumptions \ref{ass:A1}, \ref{ass:A2}, and \ref{ass:A3} hold true. Then it follows that:
    \begin{enumerate}[label=(\textit{\roman*})]
        \item The sequence $(F(x^k))_{k \in \N}$ is non increasing, and we have 
        \begin{equation}
            (\forall k \in \N)\quad F(x^k)-F(x^{k+1}) \geqslant ({\underline{\gamma}}/{2})\Vert x^{k+1}-x^k\Vert^2.
        \end{equation}
        
        \item $\sum_{k=0}^{+\infty} \Vert x^{k+1}-x^{k}\Vert^2 < +\infty$, and, in particular, $\Vert x^{k+1}-x^k\Vert \xrightarrow{k \to \infty} 0$.
    \end{enumerate}
    \end{proposition}

    \begin{proof}
    Let $k \in \N$.
         \begin{enumerate}[label=(\textit{\roman*})]
        \item By the definition of the update in \cref{alg:VBMM},
        \begin{equation*}
            Q_f(x^{k+1},x^k)+g(x^{k+1}) - \left(Q_f(x^k,x^k)+g(x^k)\right) \leqslant 0.
        \end{equation*}
        
        Moreover, by \cref{ass:A3:i}: 
        \begin{equation*}
            Q_f(x^{k+1},x^k)-f(x^{k+1}) - (\underbrace{Q_f(x^k,x^k)-f(x^k)}_{=0}) \geqslant ({\underline{\gamma}}/{2})\Vert x^{k+1}-x^k\Vert^2.
        \end{equation*} 
        
        Subtracting the second inequality from the first yields the desired sufficient decrease condition  
        \begin{equation}
            F(x^k)-F(x^{k+1}) \geqslant ({\underline{\gamma}}/{2})\Vert x^{k+1}-x^k\Vert^2.
        \end{equation}
        
        \item Summing the last inequality, and using  \cref{ass:A1:iv,ass:A3:i}:
        \begin{align*} \sum_{k=0}^K \Vert x^{k+1}-x^k \Vert^2 &\leqslant 2 \sum_{k=0}^K \frac{1}{\underline{\gamma}}(F(x^k)-F(x^{k+1})) \\
        &\leqslant \frac{2}{\underline{\gamma}}(F(x^0)-F(x^{N+1})) \\
        &\leqslant \frac{2}{\underline{\gamma}}(F(x^0)-\underline{F}) < +\infty.
        \end{align*}
        The positive-term series $\sum_{k=0}^{+\infty} \Vert x^{k+1}-x^{k}\Vert^2$, having bounded partial sums, converges. 
    \end{enumerate}
    \end{proof}

    \subsection{Relative error condition} We now derive the second condition for $(x^k)_{k \in \N}$ to be a \textit{gradient-like descent sequence} for minimizing $F$ (\cite[Definition 6.1]{BolteFirstOrder}).

    \begin{proposition}{\textbf{(Relative error condition)}}\\
    \label{pro:rel-error}
    Suppose that Assumptions \ref{ass:A1}, \ref{ass:A2}, and \ref{ass:A3} hold true. For every $k \in \N$, let
    \begin{equation}
        w^{k+1}=\nabla f(x^{k+1}) - \nabla Q_f(x^{k+1},x^k).
    \end{equation} 
    Then the following holds true: 
    \def\labelenumi{\rm (\roman{enumi})}
    \begin{enumerate}[label=({\roman*})]
        \item For every $k \in \N, \; w^{k+1} \in \partial F(x^{k+1})$;
        \item for every $k \in \N, \; \Vert w^{k+1} \Vert  \leqslant 
       {\overline{L}}
        \Vert x^{k+1}-x^k\Vert $.
    \end{enumerate}
    \end{proposition}

    \begin{proof}
    Let $k \in \N$.
        \begin{enumerate}[label=({\roman*})]
        \item Using the optimality condition: $0 \in \nabla Q_f(x^{k+1},x^k)+\partial g(x^{k+1})$, 
        \begin{equation*}
            (\exists v^{k+1}\in \partial g(x^{k+1}))\quad \; \nabla Q_f(x^{k+1},x^k)+v^{k+1}=0.
        \end{equation*}
        Adding and subtracting $\nabla f(x^{k+1})$ yields $w^{k+1} = \nabla f(x^{k+1}) + v^{k+1} \in \partial F(x^{k+1})$.
        
        \item
        By construction of the Bregman tangent majorant,
        for every $x\in \dom g$,
        \begin{equation}
            \nabla Q_f(x,x^k) =
            \nabla f(x^k)+ \nabla h_{x^k}(x) - \nabla h_{x^k}(x^k).
        \end{equation}

We have thus
        \begin{align}
            \Vert w^{k+1}\Vert &= \Vert \nabla f(x^{k+1}) - \nabla Q_f(x^{k+1},x^k) \Vert\nonumber\\
            &= \Vert \nabla f(x^{k+1})-\nabla 
            h_{x^k}(x^{k+1})- (\nabla f(x^k)-\nabla h_{x^k}(x^k)) \Vert.
        \end{align}

        By \cref{pro:obj-cv}, $x^{k+1}$ lies in a compact set, hence using~\cref{ass:A3:ii},
        \begin{equation}
            \Vert w^{k+1} \Vert \leqslant 
       \overline{L}\Vert x^{k+1}-x^k\Vert.
        \end{equation}
        
        By \cref{pro:sufficient-decrease}, we conclude that $\Vert w^{k+1} \Vert \xrightarrow{k \to \infty} 0$, and that the relative error condition is met.
    \end{enumerate}
    \end{proof}



    Note also that $(x^k)_{k \in \N}$ is bounded by \cref{pro:obj-cv}.
    Leveraging the above results, we now show that all cluster points of $(x^k)_{k \in \mathbb{N}}$ are critical points of $F$, and that the set $\Omega(x^0)$ satisfies the assumptions for \cref{thm:uniform-KL}.

    \begin{proposition}{\textbf{(Subsequential convergence)}}
    \label{pro:subseq-cv}\\
    Under Assumptions \ref{ass:A1}, \ref{ass:A2}, and \ref{ass:A3}, $\Omega(x^0)$ is a non empty compact subset of $\operatorname{crit}(F)$ and 
    \begin{equation}
        d(x^k, \Omega(x^0)) \xrightarrow{k \to \infty}0.
    \end{equation}
    Moreover, $F$ is constant and finite on $\Omega(x^0)$.
    \end{proposition}

    \begin{proof}
        The proof of this standard lemma in the context of \textit{gradient-like descent sequences} can be found in \citep[Lemma~6.1]{BolteFirstOrder}.
    \end{proof}
    
        
        
        
        
        
        
        
        

    \subsection{Main convergence result}
    We now introduce our final assumption, to establish our convergence result on the algorithm's iterates, without requiring convexity of $f$. 
    \begin{mdframed}[style=mdframedperso]
    \begin{assumption}
    \label{ass:A-KL}
    In problem \eqref{eq:mainpb}, \reviewchanges{$F$} satisfies the K\L{} property at each point of $\dom(\partial F) \allowbreak=\allowbreak \dom(\partial g)$. 
    \end{assumption}
    \end{mdframed}

    We are now ready to establish the finite‑length property of the iterates and guarantee convergence to a single critical point. By combining Assumptions \ref{ass:A1}, \ref{ass:A2}, \ref{ass:A3}, and \ref{ass:A-KL}, we can apply the uniformized K\L{} result from \cref{thm:uniform-KL} to the compact set $K = \Omega(x^0)$, yielding the existence of $\eta >0$, $\varepsilon >0$,  and $\phi \in \Phi_\eta$ such that 
    \begin{equation}
        (\forall \Bar{x}\in \Omega(x^0))\quad (\forall x
        \in \R^N)\quad
        \begin{cases}
    d(x,\Omega(x^0))<\varepsilon\\
    F(\Bar{x})<F(x)<F(\bar{x})+\eta
    \end{cases}
    \end{equation}
implies
    \begin{equation}
        \phi'(F(x)-F(\Bar{x}))\,
        d(0,\partial F(x) )\geqslant 1.
    \end{equation}

    In the following result we combine this inequality with the sufficient decrease condition established in \cref{pro:sufficient-decrease} as well as the bounds on the subgradients given by the relative error condition in \cref{pro:rel-error} to show that under our assumptions, the sequence produced by \cref{alg:VBMM} is a Cauchy sequence that converges to a critical point of $F$.

    \begin{theorem}{\textbf{(Global Convergence)}}\label{th:generalcv}\\
    Under Assumptions \ref{ass:A1}, \ref{ass:A2}, \ref{ass:A3}, and \ref{ass:A-KL}, the sequence $(x^k)_{k \in \N}$ produced by \cref{alg:VBMM} has finite length, i.e.,
    \begin{equation}
        \sum_{k=1}^{+\infty} \Vert x^{k+1}-x^k\Vert<+\infty.
    \end{equation}
    Moreover, $(x^k)_{k \in \N}$ converges to a critical point of $F$.    
    \end{theorem}

    \begin{proof}
        Combining all Assumptions \ref{ass:A1}, \ref{ass:A2}, and \ref{ass:A3} ensures that the sequence $(x^k)_{k \in \N}$ generated by \cref{alg:VBMM} satisfies all the axioms for being a \textit{gradient-like descent sequence for minimizing $F$}, as defined in \citep[Def.~6.1]{BolteFirstOrder}. Moreover, $(x^k)_{k \in \N}$ is bounded by \cref{pro:obj-cv}. The proof for the convergence result is then given in \citep[Thm~6.2]{BolteFirstOrder}.
    \end{proof}

    \subsection{Convergence rate}
    We finally derive the convergence rate of \cref{alg:VBMM}, using the following standard K\L{} result.
    
        

    \begin{theorem}{\textbf{(Convergence Rate)}}\\
    Assume that Assumptions \ref{ass:A1}, \ref{ass:A2}, \ref{ass:A3}, and \ref{ass:A-KL}, with desingularizing function $\phi: t \longmapsto c t^{1-\theta}$ with $c>0$ and $\theta \in [0,1)$ hold. Let  $x^\ast$ be the limit point of $(x^k)_{k\in\N}$. Then,
    \begin{itemize}
        \item[$\bullet$] if $\theta =0$, $(x^k)_{k\in \N}$ converges to $x^\ast$ in a finite number of iterations;
        \item[$\bullet$] if $\theta \in (0,1/2]$, then there exists $K_0 >0, C>0, \rho \in (0,1)$ such that, for every $k \geqslant K_0$, 
        \begin{equation*}
            \Vert x^k - x^\ast \Vert \leqslant C \rho^k \quad \textsc{(Linear convergence)};
        \end{equation*}
        \item[$\bullet$] if $\theta \in (1/2,1)$, then there exists $K_0 >0, C>0$ such that for all $k \geqslant K_0$, 
        \begin{equation*}
            \Vert x^k - x^\ast \Vert \leqslant C k^{-\frac{1-\theta}{2\theta - 1}} \quad\textsc{(Sub-linear convergence)}.
        \end{equation*}
    \end{itemize}
        
    \end{theorem}

    \begin{proof} A proof of this classical K\L{} result can be found in \citep[Thm~5]{AttouchBolte}.
    \end{proof}
    
\section{A unifying framework for Bregman majorization methods}\label{sec:unifying-framework}
In this section, we present several useful Bregman majorant properties as well as majorization techniques that we then apply to give an inventory of Bregman tangent majorants for the Poisson data fidelity term.
\subsection{Bregman majorants and their properties}
We refer the reader to \cref{def:BregmanMaj}, for the notion of Bregman tangent majorant.  
Let us start this section by defining the following order, allowing to compare different choices of Bregman metrics when constructing majorants.

    \begin{definition}{\textbf{(Order relation between Bregman metrics)}}\\
    \label{def:btm_order}
Let $(h_1,h_2) \in (\mathcal{L}\big(\R^N)\big)^2$. 
    We define the following tightness order relation: 
        \begin{multline}
         h_1 \preceq h_2 \Longleftrightarrow \\ \left(\forall\, (x,y) \in 
\bigl(\dom(h_1)\cap\dom(h_2)\bigr) 
\times 
\Bigl(\Int(\dom(h_1)) \cap \Int(\dom(h_2))\Bigr)
\right)
\quad D_{h_1}(x,y) \leqslant D_{h_2}(x,y).   
        \end{multline}
    \end{definition}
    We can now state the following lemma.

    \begin{lemma}{\textbf{(Bregman tangent majorants characterization from order)}}\\
    \label{lemma:bregmajorder}
        Let $f$ satisfying \cref{ass:A1:i}. 
        Let $z \in \mathcal{D}$. Assume that there exists $Q_{1,f}(\cdot,z)$, Bregman tangent majorant of $f$ at $z$, associated with $h_{1,z} \in \mathcal{L}(\R^N)$ with 
        $\mathcal{D}\subset \Int(\dom h_{1,z})$. Let $h_{2,z} \in \mathcal{L}(\R^N)$ be such that
        $\mathcal{D}\subset \Int(\dom h_{2,z})$ and
        $h_{1,z} \preceq h_{2,z}$. Then, $Q_{2,f}(\cdot,z)$
        is also a Bregman tangent majorant of $f$ at $z$, associated with $h_{2,z}$.
    \end{lemma}
        \begin{proof}
        By \cref{def:btm_order} and \eqref{eq:Bregmandef1}, $h_{1,z} \preceq h_{2,z}$ yields that, for all $x \in \mathcal{D}$, $Q_{1,f}(x,z) \leq Q_{2,f}(x,z)$, hence the result using \eqref{eq:Bregmandef2}.  
    \end{proof}
    We now provide the following lemma, that will allow to construct majorants for linear composite functions, by linearly combining majorants of each individual terms.
    \begin{lemma}
    \label{le:linBM}
    {\textbf{(Linearity in Bregman metrics)}}\\
        Let \reviewchanges{$(\alpha, \beta) \in [0,+\infty)^2 \setminus \{(0,0)\}$} and $(h_1, h_2) \in (\mathcal{L}(\R^N))^2$. For all $(x,y) \in (\dom(h_1)\cap\dom(h_2))\times (\Int  (\dom(h_1))\cap \Int(\dom(h_2))$, 
        we have $$D_{\alpha h_1+\beta h_2}(x,y) = \alpha D_{h_1}(x,y)+\beta D_{h_2}(x,y).$$
    \end{lemma}
    \begin{proof}
        Direct consequence of the bilinearity of the scalar product in \eqref{eq:Bregman}.
    \end{proof}

    
    \begin{corollary}{\textbf{(Additivity property of the Majorants)}} \\
        Let $f_1$ and $f_2$ satisfy \cref{ass:A1:i} on the open set $\mathcal{D}$
        with  Bregman tangent majorants $Q_{f_1}(\cdot,z), Q_{f_2}(\cdot,z)$ at $z\in\mathcal{D}$, associated respectively with $h_{1,z}$ and $h_{2,z}$
        with $\mathcal{D}\subset \Int(\dom h_{1,z}) \cap \Int(\dom h_{2,z})$. Let $(\alpha, \beta)\in [0,+\infty)^2$. 
        Then, the function $\alpha f_1 + \beta f_2$ satisfies \cref{ass:A1:i}, and a Bregman tangent majorant of it at $z$ is $\alpha Q_{f_1}(\cdot,z) + \beta Q_{f_2}(\cdot,z)$, associated with $\alpha h_{1,z} + \beta h_{2,z}$.
    \end{corollary}
    \begin{proof}
        Consequence of \cref{le:linBM}.
    \end{proof}
    
We will exploit this result in our constructions of majorant functions, in \cref{sec:Appli-PET}. 
    
    We now present a result regarding the construction of Bregman majorant functions, relying on second-order properties of function $f$.

\begin{proposition}\label{e:propBex}{\textbf{(Hessian Characterization)}} \\
    Let $f$ be a twice-continuously differentiable function on an open set $\mathcal{D}\subset \R^{N}$
    and let $z\in \mathcal{D}$.  Let $h_{z}$ be a twice-continuously differentiable Legendre function on $\mathcal{D}$.
    \begin{enumerate}[label=(\roman*)]
    \item\label{e:propBexi} $Q_f(\cdot,z)$ is a Bregman tangent majorant of $f$ at $z$ associated with $h_z$ if and only if
        \begin{equation}\label{e:CNSexistBmaj}
        (\forall x \in \mathcal{D})\quad
        (x-z)^\top C_{h_{z}-f}(x,z) (x-z) \ge 0,
        \end{equation}
        where $C_h$ is defined, for every function $h$ twice-differentiable on $\mathcal{D}$, as
        \begin{equation}
         (\forall x \in \mathcal{D})\quad C_{h}(x,z) = \int_{0}^1 (1-t) \nabla^2 h\big((1-t)z+tx\big) \,\dt.
        \end{equation}
    \item\label{e:propBexii}  Let $h'_{z}$ be another twice-differentiable Legendre function on $\mathcal{D}$. Then, $h_{z}$ is tighter than $h'_{z}$ if and only if
        \begin{equation}
        (\forall x \in \mathcal{D})\quad
        (x-z)^\top C_{h'_{z}-h_{z}}(x,z) (x-z) \ge 0.
        \end{equation}
    \end{enumerate}
\end{proposition}

\begin{proof} Let $z \in \mathcal{D}$.
    \begin{enumerate}[label=(\roman*)]
    \item By definition, $Q_f(\cdot,z)$ is a Bregman tangent majorant of $f$ at $z$ associated with $h_z$ if and only if,
        for every $x\in \mathcal{D}$,
        \begin{align}
            &Q_f(x,z) - f(x) \geq 0 \\
            \Leftrightarrow \quad &f(z)+\langle\nabla f(z), x-z\rangle+D_{h_{z}}(x,z)-f(x) \ge 0\nonumber\\
            \Leftrightarrow \quad &f(z)+\langle \nabla f(z), x-z\rangle+(h_z(x)-h_{z}(z)-\langle\nabla h_{z}(z),x-z\rangle)-f(x)\ge 0.
            \label{e:CNSexistBmaj0}
        \end{align}
        
        On the other hand, according to second-order Taylor formula with Laplace integral remainder, for all $x$ in $\mathcal{D}$,
        \begin{equation}
        f(x)=f(z)+\langle\nabla f(z),x-z\rangle+ (x-z)^\top C_{f}(x,z) (x-z).
        \end{equation}
        Furthermore, for all $x$ in $\mathcal{D}$,
        \begin{equation}
        h_{z}(x)=h_{z}(z)+\langle\nabla h_{z}(z),x-z\rangle+ (x-z)^\top C_{h_{z}}(x,z) (x-z).
        \end{equation}
        Plugging the last two relations in \eqref{e:CNSexistBmaj0} and using the fact that
        $C_{h_{z}-f}=C_{h_{z}}-C_{f}$ by linearity of the integral and the Hessian yields the necessary and sufficient condition \eqref{e:CNSexistBmaj}.
    \item According to \cref{def:btm_order}, $h_{z} \preceq h'_{z}$ if and only if, for every $x\in \mathcal{D}$,
        \begin{align*}
        &D_{h'_{z}}(x,z)-D_{h_{z}}(x,z) \ge 0 \\
        \Leftrightarrow \quad &(h'_z(x)-h'_{z}(z) -\langle\nabla h'_{z}(z), x-z\rangle)-(h_z(x)-h_{z}(z) -\langle\nabla h_{z}(z),x-z\rangle)\ge 0.
        \end{align*}
        The result again follows from second-order Taylor expansions with Laplace integral remainder and the convexity of $h_z$.
    \end{enumerate}
\end{proof}

We recall hereafter a result stemming from the literature of image reconstruction, for building quadratic majorant functions.

\begin{lemma}\label{le:maquad}{\textbf{(Quadratic majorization)}}\cite{ErdoganFessler1999} \\
    Let $\psi$ be a twice-continuously differentiable function on $[-\tau,+\infty)$ with $\tau>0$, and let
    \begin{equation}
       (\forall z\in [-\tau,+\infty))\quad c_{\psi,\tau}(z) = \begin{cases}
        \ddot{\psi}(-\tau) & \mbox{if $z = -\tau$}\\
        \displaystyle 2\frac{\psi(-\tau)-\psi(z)+\dot{\psi}(z)(z+\tau)}{(z+\tau)^2} & \mbox{otherwise.}
        \end{cases}
    \end{equation}
    Let us assume that $\ddot{\psi}$ is decreasing on $[-\tau,+\infty)$ and that $\forall z\in [-\tau,+\infty), c_{\psi,\tau}(z) > 0$. 
    For every $z\in [-\tau,+\infty)$,
    let
    \begin{equation}
    \label{eq:h_z_quad}
        (\forall x \in \R) \quad h_{z}(x) = \frac12 c_{\psi,\tau}(z) x^2.
    \end{equation}
    Then, 
    for every $z \in \mathcal{D} = (-\tau,+\infty)$, $Q_\psi(\cdot,z)$ given by \eqref{def:BregmanMaj} with $h_z$ defined as in \eqref{eq:h_z_quad}, is a Bregman tangent majorant of $\psi$ at $z$ associated with $h_z$. In addition, \eqref{eq:Bregmandef2} with 
    $f=\psi$ extends to $(x,z)\in [-\tau,+\infty)^2$.
\label{le:majquad}
\end{lemma}

\begin{proof} Let $z \in [-\tau,+\infty)$. Let us first notice that as $c_{\psi,\tau}$ is strictly positive-valued, $h_z$ is a twice-differentiable Legendre function. 
    We first establish the majorizing property in the open set $\mathcal{D}$.
    By performing a second-order Taylor expansion, we have for every
    $z>-\tau$:
    \begin{equation}
        \psi(-\tau) = \psi(z)+\dot{\psi}(z) (-\tau-z)+ (z+\tau)^2 C_{\psi}(-\tau,z).
    \end{equation}
    Hence,
    \begin{equation}
        C_{\psi}(-\tau,z) = \frac{\psi(-\tau)-\psi(z)+\dot{\psi}(z)(z+\tau)}{(z+\tau)^2}.
    \end{equation}
    On the other hand, since $\ddot{\psi}$ is a decreasing function,
    \begin{align}
        (\forall x \in [-\tau,+\infty))\quad C_{\psi}(x,z) = \int_{0}^1 (1-t)\,\ddot{\psi}((1-t)z+tx)\,\dt
        &\le  \int_{0}^1 (1-t)\,\ddot{\psi}((1-t)z -t\tau)\,\dt \nonumber\\
        &= C_{\psi}(-\tau,z)\nonumber\\
        &= 2C_{\psi}(-\tau,z) \int_{0}^1 (1-t)\,\dt\nonumber\\
        & = 2C_{\psi}(-\tau,z) C_{(\cdot)^2/2}(x,z). 
    \end{align}
    When $z=-\tau$, using similar arguments,
    \begin{align}
        (\forall x \in [-\tau,+\infty))\quad C_{\psi}(x,-\tau) = \int_{0}^1 (1-t)\,\ddot{\psi}(tx-\tau(1-t))\,\dt
        &\le  \int_{0}^1 (1-t)\,\ddot{\psi}(-\tau)\,\dt \nonumber\\
        & = \ddot{\psi}(-\tau) C_{(\cdot)^2/2}(x,z). 
    \end{align}
    In summary, 
    \begin{equation}
        C_{\frac{1}{2}c_{\psi,\tau}(z)(\cdot)^2-\psi}(x,z) = C_{c_{\psi,\tau}(z)(\cdot)^2/2}(x,z)-C_{\psi}(x,z)\ge 0,
    \end{equation} 
    which, according to 
\cref{e:propBex}\ref{e:propBexi} shows that
$Q_\psi(\cdot,z)$ is a Bregman tangent majorant of $\psi$ at
$z$ associated with $h_z$.
The majorization property
\eqref{eq:Bregmandef2} 
with $f=\psi$ extends by continuity to $(x,z)\in [-\tau,+\infty)^2$.
\end{proof}

We conclude this subsection with the following lemma, consequence of Jensen's inequality for convex functions (see \cite[Lemma 1.9]{chouzenoux:hal-04250055}\cite{Hunter2004}, for similar results, although not formulated under the Bregman majorizing framework).  

\begin{lemma}\label{le:majlog}{\textbf{(Logarithmic separable majorization)}}\\ 
    Let $\theta = (\theta_{n})_{1\le n \le N}\in [0,+\infty)^N\setminus\{0_N\}$, $(\delta_{1},\delta_{2})\in [0,+\infty)^2$ and let
    \begin{equation}
    \label{eq:psilogjensen}
        (\forall x \in \mathcal{D})\quad \psi(x) = -\ln (\theta^\top x + \delta_{1}+\delta_{2})
    \end{equation}
    where $\mathcal{D}=(-\zeta\delta_1, +\infty)^N$ with $\zeta = 1/\sum_{n=1}^N \theta_{n}$.
    For every $z = (z_{n})_{1\le n \le N} \in \mathcal{D}$,
    let
    \begin{equation} \label{eq:Elog}
        (\forall x = (x_{n})_{1\le n \le N} \in \mathcal{D})\quad  h_{z}(x)= -\sum_{n=1}^N \frac{\theta_{n}(z_{n}+\zeta \delta_{1})}{\theta^\top z+ \delta_{1}+\delta_{2}} \ln(x_{n}+\zeta \delta_{1}).
    \end{equation}
    Then, for all $z\in\mathcal{D}$, function $Q_\psi(\cdot,z)$ with $h_z$ given in \eqref{eq:Elog}, is a Bregman tangent majorant of $\psi$ at $z$.
\end{lemma}

\begin{proof} Let $z \in \mathcal{D}$. First notice that $h_z$, as defined in \eqref{eq:Elog}, is a Legendre function, and that it is also twice differentiable on $\mathcal{D}$. Then, using \eqref{eq:psilogjensen}, for every $x = (x_{n})_{1\le n \le N} \in \mathcal{D}$,
    \begin{align}
        \psi(x) &= -\ln\left(\sum_{n=1}^N \theta_{n}(x_{n}+\zeta \delta_{1})+\delta_{2}\right)\nonumber\\
        & = \psi(z)-\ln\left(\sum_{n=1}^N \frac{\theta_{n}(z_{n}+\zeta \delta_{1})}{\theta^\top z+\delta_{1}+\delta_{2}}\frac{x_{n}+\zeta \delta_{1}}{z_{n}+\zeta \delta_{1}}+\frac{\delta_{2}}{\theta^\top z+\delta_{1}+\delta_{2}}\right)\nonumber\\
        & \le \psi(z)- \sum_{n=1}^N \frac{\theta_{n}(z_{n}+\zeta \delta_{1})}{\theta^\top z+\delta_{1}+\delta_{2}} \ln\left(\frac{x_{n}+\zeta \delta_{1}}{z_{n}+\zeta \delta_{1}}\right)
        - \frac{\delta_{2}}{\theta^\top z+\delta_{1}+\delta_{2}} \ln 1\label{e:inqJenmajlog}\\
        & = \psi(z)-\frac{\theta^\top (x-z)}{\theta^\top z +\delta_{1}+\delta_{2}}+h_{z}(x)-h_{z}(z)+\sum_{n=1}^N \frac{\theta_{n} (x_{n}-z_{n})}{\theta^\top z+\delta_{1}+\delta_{2}}\nonumber\\
        & = \psi(z)+\langle\nabla \psi(z),x-z\rangle+D_{h_{z}}(x,z)\label{eq:prooflogjensen},
    \end{align}
    where Jensen's inequality is used to get \eqref{e:inqJenmajlog}. 
    This yields the sought majorant.
\end{proof}

\subsection{A family of separable Bregman majorants for Poisson data fidelity term}\label{subsec:family-maj}
We now apply the previous constructive results to design Bregman tangent majorants for a logarithmic composite function, that is encountered as a Poisson data fidelity term in imaging inverse problems \cite{BerteroBOOK}, as explored in our \cref{sec:Appli-PET}. Other application examples for these majorizing results are barrier function minimization, in interior point methods \cite{CorbineauPIPA,ChouzenouxBarrier}, and entropy maximization problems in statistical signal processing \cite{ChouzenouxRMN}.
\\
For every $m\in \llbracket 1,M\rrbracket$, let $H_m = (H_{m,n})_{1\le n\le N}$ be the $m$-th line of $H\in\R^{M \times N}_{+}$, and $y=(y_m)_{1\le m \le M} \in \N^M$. \reviewchanges{We assume that $\sum_{n=1}^N H_{m,n}>0$ for every
$m\in\llbracket1,M\rrbracket$ and that $H^\top y>0$, which are standard
assumptions when $H$ models a tomographic projection operator. The above assumptions ensure that every coefficient appearing in the
definition of the proposed majorants is strictly positive. Therefore,
all the associated logarithmic and quadratic generators are Legendre on
their respective domains. Furthermore,} 
let $b_{m}\in (0,+\infty)$, let  $\zeta_{m}=1/\sum_{n=1}^N H_{m,n}$, and let $\rho = \min_{m \in \llbracket 1,M\rrbracket} \zeta_m b_m$. 

We introduce
\begin{align}\label{def:notation_somme_log_hx_b}
    &(\forall x \in (-\rho,+\infty)^N)\quad \ell(x) = 
    \sum_{m=1}^M \ell_m(x),\\
\text{where}\quad 
& (\forall m \in \llbracket 1,M\rrbracket) \quad \ell_m(x)= -y_{m}\ln (H_m x + b_{m}).
\end{align}
Function \eqref{def:notation_somme_log_hx_b} is twice differentiable on $(-\rho,\infty)^N$. In the following, we present a collection of Legendre functions, each twice differentiable on the interior of their domain, providing a separable Bregman tangent majorant approximation for $\ell$ on an open set $\mathcal{D}$ to be specified. The proposed Legendre functions are indexed,
for an easier comparison between them. The end of the section is dedicated to the positioning of the obtained majorizing results, with respect to existing formulations in the literature of Poisson image reconstruction.\\
In the following,
$(x_n)_{n\in \N}$ and $(z_n)_{n\in \N}$ denote the components of vectors $x$ and $z$, respectively.

\subsubsection{Log-shift Majorants}

Let us first introduce a family of logarithmic separable majorants, relying on \cref{le:majlog}.

\begin{proposition}\label{prop:logshift_majorants}
Let $\ell$ be defined by \eqref{def:notation_somme_log_hx_b}. The following Legendre functions are associated with separable Bregman tangent majorants of $\ell$ on a certain open set $\mathcal{D}\subseteq \dom(\ell)$.

 \begin{description}
    \item[(\textbf{Maj-$\majO$})] Let $\mathcal{D}=(-\rho,+\infty)^N$. Let $z \in \mathcal{D}$ and \begin{align}
    \label{eq:h0}
    &(\forall x 
    \in \mathcal{D}) \quad h_{\majO,z}(x)= -\sum_{n=1}^N a_{\majO,n}(z) \ln (x_{n}+\rho)\\
    \label{eq:a0}
    \text{with}\quad
            &(\forall n \in \llbracket 1,N\rrbracket )\quad 
            a_{\majO,n}(z) = \sum_{m=1}^M y_{m}  \frac{H_{m,n}(z_{n}+\zeta_{m}b_{m})}{H_m z+ b_{m}}.
\end{align}

\item[(\textbf{Maj-$\majII$})]  Let $\mathcal{D}=(-\rho,+\infty)^N$. Let $z \in \mathcal{D}$ and \begin{align}
    &(\forall x 
    \in \mathcal{D}) \quad h_{\majII,z}(x)= -\sum_{n=1}^N a_{\majII,n} \ln (x_{n}+\rho)  \label{eq:a2}\\
\text{with}\quad
            &(\forall n \in \llbracket 1,N\rrbracket )\quad 
            a_{\majII,n} = \sum_{m=1}^M y_{m}\mathbbm{1}_{H_{m,n}\neq 0}.
        \end{align}

    \item[(\textbf{Maj-$\majIII$})]  \label{prop:listmajiii} Let $\mathcal{D}=(-\rho,+\infty)^N$. Let $z\in \mathcal{D}$ and
            \begin{equation}
            (\forall x \in \mathcal{D})\quad
                h_{\majIII,z}(x)= \sum_{n=1}^N 
                a_{\majO,n}(z) 
                \varphi_n(x_n)          
            \end{equation}
            where $ (a_{\majO,n})_{1\le n\le N}$ is defined in \eqref{eq:a0}
            and, for every $n\in \llbracket 1,N\rrbracket $,
    \begin{equation}
    \label{e:defvarphin}
        (\forall \xi \in (-\rho,+\infty))\quad
        \varphi_{n}(\xi) = 
        \left(\mathbbm{1}_{\xi\ge z_n} \Big(\frac{(\xi-z_{n})^2}{2(z_{n}+\rho)^2}-\frac{\xi-z_{n}}{z_{n}+\rho}\Big)- \mathbbm{1}_{\xi<z_n} \ln \Big(\frac{\xi+\rho}{z_{n}+\rho}\Big)\right).
    \end{equation}
            
    \item[(\textbf{Maj-$\majVIII$})] \label{prop:listmajviii} Let $\mathcal{D}=(-\mu,+\infty)^N$ with $\mu \in [0,\rho]$. Let $z \in \mathcal{D}$ and 
    \begin{align} \label{eq:h8}
    &(\forall x 
    \in
    \mathcal{D})\quad
            h_{\majVIII,z}(x)= -\sum_{n=1}^N a_{\majVIII,n}(z) \ln (x_{n}+\mu)\\\
        \text{with}\quad
            &(\forall n \in \llbracket 1,N\rrbracket )\quad 
            a_{\majVIII,n}(z) = \sum_{m=1}^M y_{m}  \frac{H_{m,n}(z_{n}+\mu)}{H_m z+ b_{m}}. \label{eq:a8}
        \end{align}
\end{description}
In addition, the following order relations hold, for every $z \in (-\rho,+\infty)^N$:
    \begin{enumerate}[label=(\roman*)]        
    \item\label{prop:ordre_02} $h_{\majO,z} \preceq h_{\majII,z}$;
    \item\label{prop:ordre_03} $h_{\majO,z}\preceq h_{\majIII,z}$;
    \item\label{prop:ordre_80} if $\mu=\rho$, $h_{\majVIII,z} \preceq h_{\majO,z}$.
    \end{enumerate}
\end{proposition}

\begin{proof}
 \begin{description} 
    \item[(\textbf{Maj-$\majO$})] Let $\mathcal{D}=(-\rho,+\infty)^N$. We apply \cref{le:majlog} to each component $(\ell_m)_{m\in \llbracket 1,M\rrbracket}$ of $\ell$ (set $\theta = H_m$, $\delta_{1}=b_{m}$, and $\delta_{2}=0$). Then, using the additivity property, we deduce that a Bregman tangent majorant of $\ell$ at $z$ is associated with
    \begin{equation}
        (\forall x \in \mathcal{D}=(-\rho,+\infty)^N)\quad
        \widetilde{h}_{\majO,z}(x)= -\sum_{m=1}^M \sum_{n=1}^N y_{m}\frac{H_{m,n}(z_{n}+\zeta_{m}b_{m})}{H_m z+b_{m}} \ln (x_n+\zeta_{m} b_{m}).
    \end{equation}
    Then, for every $x\in \mathcal{D}$,
    \begin{align}
        (x-z)^\top C_{\widetilde{h}_{\majO,z}}(x-z)\nonumber
        =& \sum_{m=1}^M \sum_{n=1}^N y_{m}\frac{H_{m,n}(z_{n}+\zeta_{m}b_{m})}{H_m z+b_{m}} 
        \int_{0}^1 \frac{(1-t)\, (x_{n}-z_{n})^2}{((1-t) z_n+t x_{n}+\zeta_{m} b_{m})^2}\, \dt\label{e:needforquad}\\
        \le & \sum_{m=1}^M \sum_{n=1}^N y_{m}\frac{H_{m,n}(z_{n}+\zeta_{m}b_{m})}{H_m z+b_{m}} 
        \int_{0}^1 \frac{(1-t)\, (x_{n}-z_{n})^2}{((1-t) z_n+t x_{n}+\rho)^2}\, \dt\nonumber\\
        =& \sum_{n=1}^N a_{\majO,n}(z) \int_{0}^1 \frac{(1-t)\, (x_{n}-z_{n})^2}{((1-t) z_n+t x_{n}+\rho)^2}\,\dt
        = (x-z)^\top C_{h_{\majO,z}}(x-z).
    \end{align}
    This shows that $\widetilde{h}_{\majO,z} \preceq h_{\majO,z}$ and it follows from \cref{lemma:bregmajorder} that a Bregman tangent majorant of $\ell$ at $z$ is associated with $h_{\majO,z}$.
\item[(\textbf{Maj-$\majII$})] Let $\mathcal{D}=(-\rho,+\infty)^N$. For every $x\in \mathcal{D}$,
    \begin{align}
        (x-z)^\top C_{h_{\majO,z}}(x-z)
        \le & \sum_{n=1}^N \sum_{m=1}^M  y_{m}\mathbbm{1}_{H_{m,n}\neq 0}
        \int_{0}^1 \frac{(1-t)\, (x_{n}-z_{n})^2}{((1-t) z_n+t x_{n}+\rho)^2}\, \dt\nonumber\\
        =& \sum_{n=1}^N a_{\majII,n} \int_{0}^1 \frac{(1-t)\, (x_{n}-z_{n})^2}{((1-t) z_n+t x_{n}+\rho)^2}\,\dt
        = (x-z)^\top C_{h_{\majII,z}}(x-z).
    \end{align}
    Hence $h_{\majO,z} \preceq h_{\majII,z}$, proving \ref{prop:ordre_02}.  Moreover, by \cref{lemma:bregmajorder}, a Bregman tangent majorant of $\ell$ at $z$ is associated with $h_{\majII,z}$.
\item[(\textbf{Maj-$\majIII$})] Let $\mathcal{D}=(-\rho,+\infty)^N$ 
and $x\in \mathcal{D}$.
    For every $n\in \llbracket 1,N\rrbracket $ such that $x_{n} \ge z_{n}$,
\begin{align}\label{e:majlogquad3}
        C_{-\ln(\cdot+\rho)}(x_{n},z_{n}) = \int_{0}^1 \frac{(1-t)}{((1-t) z_n+t x_{n}+\rho)^2}\,\dt
        &\le \frac{1}{2(z_{n}+\rho)^2}  \nonumber\\
        & 
        = \frac{1}{(z_{n}+\rho)^2}  C_{(\cdot)^2/2}(x_{n},z_{n}).
    \end{align}
    For every $n\in \llbracket 1,N\rrbracket $, let $\varphi_{n}$ be defined by    
    \eqref{e:defvarphin}.
    Using  \eqref{e:majlogquad3} leads to
\begin{align}
(x-z)^\top C_{h_{\majO,z}}(x-z) 
&= \sum_{n=1}^N a_{\majO,n}(z) \int_{0}^1 \frac{(1-t)\, (x_{n}-z_{n})^2}{((1-t) z_n+t x_{n}+\rho)^2}\,\dt \nonumber\\
&\le \sum_{n=1}^N a_{\majO,n}(z) \left(\mathbbm{1}_{x_{n}\ge z_{n}} \left(\frac{1}{(z_{n}+\rho)^2}  C_{(\cdot)^2/2}(x_{n},z_{n})\right)\right. \nonumber\\
&\qquad \left. + \mathbbm{1}_{x_{n}< z_{n}} \left(C_{-\ln(\cdot+\rho)}(x_{n},z_{n})\right)\right) (x_{n}-z_{n})^2 \label{eq:endproofmajIII}\\
&= \sum_{n=1}^N a_{\majO,n}(z) C_{\varphi_{n}}(x_{n},z_{n}) (x_{n}-z_{n})^2 = (x-z)^\top C_{h_{\majIII,z}}(x-z), \nonumber
\end{align}
    which shows that ${h_{\majO,z}} \preceq {h_{\majIII,z}}$, hence proving \ref{prop:ordre_03}. Consequently, a Bregman tangent majorant of $\ell$ at $z$ is associated with $h_{\majIII,z}$ according to \cref{lemma:bregmajorder}.
    \item[(\textbf{Maj-$\majVIII$})] Let $\mathcal{D}=(-\mu,+\infty)^N$. We have, for every $x \in\mathcal{D}$ and $m \in \llbracket 1,M\rrbracket$,
            \begin{equation} 
                \ln (H_m x+b_m)  = \ln (H_m x + \alpha_m b_{m} + (1-\alpha_m) b_m),
            \end{equation}
            where 
             $\alpha_{m} = \frac{\mu}{\zeta_{m} b_{m}}$ and $\mu \in [0,\rho]$.
            According to \cref{le:majlog} (set $\theta = H_m$, $\delta_{1}=\alpha_m b_{m}$, and $\delta_{2}=(1-\alpha_{m})b_{m}$), a Bregman tangent majorant of $\ell$ at $z$ is associated with
            \begin{align}
                (\forall x \in \mathcal{D}) \quad
                h_{\majVIII,z}(x) 
                & = -\sum_{n=1}^N \sum_{m=1}^M y_{m}\frac{H_{m,n}(z_{n}+\mu)}{H_m z+b_{m}} \ln (x_n+\mu).
            \end{align}
            Furthermore, if $\mu=\rho$, then $\mathcal{D}=(-\rho,+\infty)^N$, and for every $x\in\mathcal{D}$,
    \begin{align}
        (x-z)^\top C_{h_{\majVIII,z}}(x-z)\nonumber
        &=  \sum_{n=1}^N \sum_{m=1}^M  y_{m}\frac{H_{m,n}(z_{n}+\rho)}{H_m z+b_{m}} 
        \int_{0}^1 \frac{(1-t)\, (x_{n}-z_{n})^2}{((1-t) z_n+t x_{n}+\rho)^2}\, \dt\nonumber\\
        & \leq \sum_{m=1}^M \sum_{n=1}^N y_{m}\frac{H_{m,n}(z_{n}+\zeta_{m}b_{m})}{H_m z+b_{m}} 
        \int_{0}^1 \frac{(1-t)\, (x_{n}-z_{n})^2}{((1-t) z_n+t x_{n}+\rho)^2}\, \dt\nonumber\\
        &= (x-z)^\top C_{h_{\majO,z}}(x-z).
    \end{align}
    Hence if $\mu=\rho$, $h_{\majVIII,z} \preceq h_{\majO,z}$, showing \ref{prop:ordre_80}.
    \end{description}
\end{proof}

\subsubsection{Log-0 Majorants}

Let us now introduce another family of logarithmic separable majorants, with barrier at $0$.

\begin{proposition}
\label{prop:log0majorants}

Let $\ell$ be defined by \eqref{def:notation_somme_log_hx_b}. The following Legendre functions are associated with separable Bregman tangent majorants of $\ell$ at $z \in \mathcal{D}$ on the open set $\mathcal{D}= (0,+\infty)^N$.

\begin{description}
    \item[(\textbf{Maj-$\majI$})] 
Let 
\begin{equation}
 \label{eq:h1}
    (\forall x 
    \in (0,+\infty)^{N}) \quad h_{\majI,z}(x)= - \sum_{n=1}^N a_{\majO,n}(z) \ln x_{n}
\end{equation}
with
$a_{\majO,n}$ defined by \eqref{eq:a0}.
    \item[(\textbf{Maj-$\majVI$})]
    Let 
\begin{align}
    &(\forall x 
    \in (0,+\infty)^{N}) \quad h_{\majVI,z}(x)= - \sum_{n=1}^N a_{\majVI,n}(z) \ln x_{n} \label{eq:majh6}\\
\text{with}\quad 
&       (\forall n \in \llbracket 1,N\rrbracket )\quad 
        a_{\majVI,n}(z) = \sum_{m=1}^M y_{m}\frac{H_{m,n}z_{n}}{H_m z+ b_{m}}. \label{eq:maj6a6}
    \end{align}
\end{description}
In addition, the following order relations hold:
    \begin{enumerate}[label=(\roman*)]
        \item\label{prop:ordre_01} $h_{\majO,z} \preceq h_{\majI,z}$; 
        \item $h_{\majVI,z} \preceq h_{\majI,z}$. \label{prop:ordre_61}
    \end{enumerate}
\end{proposition}

\begin{proof}
Proof similar to \cref{prop:logshift_majorants}, detailed in supplementary material \cref{annex:prooflog0}.
\end{proof}

\subsubsection{Quadratic majorants}

Finally, let us present a family of quadratic separable majorants.

\begin{proposition}\label{prop:quadratic_majorants}
Let $\ell$ be defined by \eqref{def:notation_somme_log_hx_b} and let $\tau\in(0,\rho)$. The following Legendre functions yield separable quadratic Bregman tangent majorants of $\ell$  at $z\in \mathcal{D} = (-\tau,+\infty)^N$.
 \begin{description}
 \item[(\textbf{Maj-$\majV$})]
 Let
 \begin{align}
    &(\forall x \in \R^N) \quad h_{\majV,z}(x)= \frac{1}{2} \sum_{n=1}^N a_{\majV,n}(z)x_{n}^2\\
\text{with}\quad
                &(\forall n \in \llbracket 1,N\rrbracket) \quad a_{\majV,n}(z) = \sum_{m=1}^M y_{m}  \frac{H_{m,n}(z_{n}+\zeta_{m}b_{m})}{H_m z+ b_{m}} c_\tau(z_{n},\zeta_{m}b_{m}),
            \end{align}
            where, for every $\eta \in (\tau,+\infty)$,
            \begin{equation}\label{e:defcm}
                (\forall \xi \in (-\tau,+\infty))
                \quad
                c_\tau(\xi,\eta) = 
                \displaystyle-\frac{2}{(\xi+\tau)}
                \left(\frac{1}{(\xi+\tau)}
                \ln\Big(\frac{\eta-\tau}{\xi+\eta}\Big)+\frac{1}{\xi+\eta}
                \right). 
            \end{equation}
             \item[(\textbf{Maj-$\majVp$})] \label{prop:listmajivbis}
 Let
 \begin{align}
  \label{eq:h5p}
    &(\forall x 
    \in \R^N) \quad h_{\majVp,z}(x)= \frac{1}{2} \sum_{n=1}^N a_{\majVp,n}(z)x_{n}^2\\
\text{with}\quad 
                &(\forall n \in \llbracket 1,N\rrbracket) \quad a_{\majVp,n}(z) = \sum_{m=1}^M y_{m}  \frac{H_{m,n}(z_{n}+\zeta_{m}b_{m})}{H_m z+ b_{m}} c_\tau(z_{n},\rho).
                \label{eq:a5p}
            \end{align}
         \item[(\textbf{Maj-$\majVII$})] \label{prop:listmajv} 
 Let
 \begin{align}
    &(\forall x 
    \in \R^N) \quad h_{\majVII,z}(x)= \frac{1}{2} \sum_{n=1}^N a_{\majVII,n}(z)x_{n}^2\\
\text{with}\quad       
                &(\forall n \in \llbracket 1,N\rrbracket )\quad a_{\majVII,n}(z) = \sum_{m=1}^My_m\frac{H_{m,n}}{\zeta_m}c_\tau(H_mz,b_m).
            \end{align}
    \end{description}
In addition, the following order relations hold:
    \begin{enumerate}[label=(\roman*)]
        \item\label{prop:ordre_55'} $h_{\majV,z} \preceq h_{\majVp,z}$; 
        \item \label{prop:listmajb}  if, for every $m\in \llbracket 1,M\rrbracket$, $\zeta_{m} b_{m}= \rho$, then $h_{\majIII,z} \preceq h_{\majV,z}$. \label{prop:ordre_35}
    \end{enumerate}
\end{proposition}
\begin{proof}
\begin{description}        
    \item[(\textbf{Maj-$\majV$})] According to \eqref{e:needforquad}, for every $x\in (-\rho,+\infty)^N$,
    \begin{equation}
        (x-z)^\top C_{\widetilde{h}_{\majO,z}}(x-z)
        = \sum_{m=1}^M \sum_{n=1}^N y_{m}\frac{H_{m,n}(z_{n}+\zeta_{m}b_{m})}{H_m z+b_{m}} 
        (x_{n}-z_{n})^2 C_{- \ln(\cdot+\zeta_{m}b_{m})} (x_{n},z_{n}) .
    \end{equation}
    For every $m\in \llbracket 1,M\rrbracket$, set $\psi = - \ln(\cdot+\zeta_{m}b_{m})$. Since 
    $\ddot{\psi} = (\cdot+\zeta_{m}b_{m})^{-2}$ is decreasing and $\tau<\rho$, it follows from \cref{le:maquad} 
    that, for every $x\in\mathcal{D}$, $n\in \llbracket 1,N\rrbracket $,
    \begin{equation}\label{e:maqquad}
        C_{- \ln(\cdot+\zeta_{m}b_{m})} (x_{n},z_{n})  \le \frac12 c_\tau(z_{n},\zeta_{m}b_{m}).
    \end{equation}
    We deduce from the previous two relations that
    \begin{align}
        (x-z)^\top C_{\widetilde{h}_{\majO,z}}(x-z) \le 
        \frac12 \sum_{n=1}^N a_{\majV,n}(z)  (x_{n}-z_{n})^2 
        = &
        (x-z)^\top C_{h_{\majV,z}}(x-z).
    \end{align}
    Thus $\widetilde{h}_{\majO,z} \preceq h_{\majV,z}$, and a Bregman tangent majorant of $\ell$ at $z$ is associated with $h_{\majV,z}$ (by Lemma~\ref{lemma:bregmajorder}).
\item[(\textbf{Maj-$\majVp$})]  Based on the proof of \cref{le:maquad}, for every $n\in \llbracket 1,N\rrbracket $ and $m\in \llbracket 1,M\rrbracket$,
    \begin{align}
        c_\tau(z_{n},\zeta_{m}b_{m}) 
        = 2\,C_{- \ln(\cdot+\zeta_{m}b_{m})}(-\tau,z_{n})
        & 
        = 2\int_{0}^1 \frac{1-t}{((1-t)z_{n}-t\tau+\zeta_{m}b_{m})^2}\,\dt
        \\
        &\le 2\int_{0}^1 \frac{1-t}{((1-t)z_{n}-t\tau+\rho)^2}\,\dt = c_\tau(z_{n},\rho).
    \end{align}
    We deduce  that
    \begin{align}
        (x-z)^\top C_{h_{\majV,z}}(x-z)= &\; \frac12 \sum_{n=1}^N a_{\majV,n}(z)  (x_{n}-z_{n})^2 \nonumber\\
        \le &\; \frac12 \sum_{n=1}^N a_{\majVp,n}(z)  (x_{n}-z_{n})^2 
        = (x-z)^\top C_{h_{\majVp,z}}(x-z).
    \end{align}
    Thus $h_{\majV,z} \preceq h_{\majVp,z}$, showing \ref{prop:ordre_55'}, and a Bregman tangent majorant of $\ell$ at $z$ is associated with $h_{\majVp,z}$ (by \cref{lemma:bregmajorder}).
\item[(\textbf{Maj-$\majVII$})]  For every $x\in \mathcal{D}$ and for every $m\in \llbracket 1,M\rrbracket$,
    \begin{align}
        (x-z)^\top C_{\ell_m}(x,z) (x-z)\nonumber
        = &\; y_{m}\int_{0}^1 
        \frac{(1-t)\, \left(H_m
        (x-z)\right)^2}{(H_m
        ((1-t)z+t x)+b_{m})^2}\,
        \dt\nonumber\\
        = &\; y_{m}\left(H_m
        (x-z)\right)^2 C_{- \ln(\cdot+b_{m})} (H_m x,H_m z).
    \end{align}
    It follows from \cref{le:maquad} (set $\psi = - \ln(\cdot+b_{m})$)
    that
    \begin{align}
        &(x-z)^\top C_{\ell_m}(x,z) (x-z)
        \le 
        \frac{y_{m}}{2}\left(H_m
        (x-z)\right)^2 c_\tau(H_m z,b_{m}).
    \end{align}
    Applying Jensen's inequality yields
    \begin{align}
        (x-z)^\top C_{\ell_m}(x,z) (x-z)
        \nonumber
        \le &\; \frac{y_{m}}{2} \left(\sum_{n'=1}^N H_{m,n'}\right)^2
        \left(\sum_{n=1}^N \frac{H_{m,n}(x_{n}-z_{n})}{\sum_{n'=1}^N H_{m,n'}}\right)^2 c_\tau(H_m z,b_{m})\nonumber\\
        \le &\;  \frac{y_{m}}{2}\left(\sum_{n'=1}^N H_{m,n'}\right) c_\tau(H_m z,b_{m})  \sum_{n=1}^N H_{m,n} (x_{n}-z_{n})^2.
    \end{align}
    We deduce that
    \begin{align}
        (x-z)^\top C_{\ell}(x,z) (x-z)
        \le
        &\;  \frac12\sum_{n=1}^N a_{\majVII,n}(z) (x_{n}-z_{n})^2
        =
        (x-z)^\top C_{h_{\majVII,z}}(x,z) (x-z),
    \end{align}
    and the result follows from \cref{e:propBex}\ref{e:propBexi}. 

     Proof for \ref{prop:ordre_35}, analogous to that of \ref{prop:ordre_55'}, can be found in supplementary material \cref{annex:proof35}.
    \end{description}
\end{proof}

\subsubsection{Relation with majorant construction in the literature}
\label{sec:maj_litt}

There is a significant amount of literature addressing the construction of separable majorant functions for imaging inverse problems under Poisson noise, and more generally for Kullback-Leibler minimization problems. This methodology likely originates from the EM-based approach~\cite{Dempster1977} which gave rise to the ML-EM method \cite{SheppVardi}, equivalently known as the Richardson-Lucy algorithm \cite{Richardson1972,Lucy1974}, for minimizing the Poisson data fidelity term under positivity constraints. Convergence proofs for this algorithm can be found in \cite{Natterer2001, Bertero2022}. The majorant function \textbf{Maj}-$\majVI$ is identical to the one used to build the aforementioned ML-EM algorithm, and our proof follows the same construction mechanism as the one detailed for instance in \cite{DePierro1993,Bertero2022}. Due to its form, which includes a logarithmic function, this majorant function promotes positivity of the iterates of the resulting MM scheme. It has subsequently been used in many works in the field of image restoration: for instance, let us cite \cite{Wang2012} for maximum-a-posteriori PET image reconstruction, and \cite{Gong2019,Mehranian2021} for AI-based PET reconstruction methods. As emphasized in several works, including \cite{BONETTINI2021113192,Zanella_2009,Lanteri2002}, the construction of the majorant function \textbf{Maj}-$\majVI$ is also tightly linked to split gradient algorithms, that propose multiplicative updates for positivity-constrained optimization problems. Similar majorant constructions are also at the core of the alternating multiplicative algorithms for nonnegative matrix factorization \cite{lee99,FevotteBetaDiv}. The approach we use in \textbf{Maj}-$\majIII$, decomposing the majorizing terms depending on derivative signs, was also employed in \cite{ChouzenouxBarrier} in the context of barrier function majorization and in \cite{BONETTINI2021113192} for building accelerating variable metrics in proximal algorithms. \reviewchanges{Majorant \textbf{Maj}-\majII{}, based on a constant logarithmic metric, recovers an instance of the Bregman Proximal Gradient (BPG) algorithm.}
Paraboloidal (i.e., quadratic) surrogates for the Poisson data fidelity term, such as those presented in \cref{prop:quadratic_majorants}, were originally proposed  in \cite{Fessler1998} and applied to PET dynamic imaging \cite{Wang2009}. The quadratic majorant \textbf{Maj}-$\majV$ is retrieved by applying the same technique as in \cite{Fessler1998} after Jensen's inequality. 
\reviewchanges{The iteration-dependent quadratic metrics associated with majorants \textbf{Maj}-\majVp{} and \textbf{Maj}-$\majV$ naturally yield variable metric forward-backward schemes, similarly to \cite{ChouzenouxPesquetRepetti2015,BONETTINI2021113192}.}
Majorant \textbf{Maj}-$\majVII$ can be traced back to the separable paraboloid algorithm of \cite{Wang2009}. Similar constructions for quadratic separable surrogates have also been used to approximate non-separable regularization terms, for instance in a patch-based regularization in PET imaging \cite{Wang2012}. 

The Log-shift majorant \textbf{Maj}-$\majO$ from \cref{prop:logshift_majorants} was introduced in \cite{FesslerHero1995} to derive a penalized ML-EM algorithm accounting for a non-negative background term, building upon the statistical framework of generalized EM~\cite{Fessler1993}. The Log-shift construction was shown in \cite{FesslerHero1995} to yield significant acceleration of the reconstruction algorithm compared to the Log-0 majorant \textbf{Maj}-$\majVI$. From a statistical viewpoint, the Log-shift approach exploits a more informative complete-data space than the Log-0 majorant, for the estimates related to background events.

In a nutshell,  our constructive approach allows us to recover many majorants found in the literature. \reviewchanges{To the best of our knowledge, this section presents the first unified and constructive presentation of the many majorizing constructions, including new ones, by relying on the framework of Bregman divergence. As a consequence of our analysis, the convergence theorem from \cref{sec:convergence-analysis} can be applied to several existing ad hoc MM methods for Poisson imaging, and allows us to build convergent extensions of those incorporating non-convex and non-smooth penalty terms}. Moreover, our propositions provide the first formal comparison of the various majorant approximations, thanks to the introduced Bregman-based order relation.

\section{Application to PET Reconstruction}
\label{sec:Appli-PET}
Building on the theoretical VBMM framework and Bregman Majorizing techniques presented in previous sections, we now address the challenging problem of Positron Emission Tomography (PET) image reconstruction.

    \subsection{Problem statement} 
    
In PET, the goal is to recover an image of the radiotracer activity concentration $x \in [0,+\infty)^N$ from a measured noisy sinogram $y \in \N^M$, which is modeled by a Poisson inverse problem. The forward model reads
\begin{equation}\label{model:inverse-poisson}
    (\forall m \in \llbracket 1,M\rrbracket) \quad y_m = \mathcal{P}(H_m \overline{x} + b_m)
\end{equation}
where $H = (H_m)_{1\leq m\leq M}\in \R^{M \times N}_{+}$ is the system matrix describing the PET scanner's geometry and physics, \(b = (b_m)_{1\leq m \leq M}\in (0,+\infty)^M\) is a strictly positive background term accounting for factors like scattering and random coincidences, and \(\mathcal{P}(\cdot)\) denotes the Poisson noise corruption process, stemming from the stochastic nature of the detected photon counts. Reconstruction of $\overline{x}$ from  $y=(y_m)_{1\le m \le M}$ can be formulated as an optimization problem in the general framework \eqref{eq:mainpb}, by setting 
\begin{equation}
(\forall x \in \R^N)\quad
    f(x) = \mathcal{L}(x) + \mathcal{R}(x) \quad \text{and} \quad g(x) = \iota_{\mathcal{D}_0}(x) 
    \label{eq:recons_tep_in_kl_framework}
\end{equation}
where $\mathcal{L}$,  defined on its domain
as $\mathcal{L}\colon x \mapsto \textsc{KL}(y,Hx+b) = -\sum_{m=1}^M (y_{m}\ln (H_m x + b_{m}) - H_m x)$,
is the Kullback-Leibler divergence between the predicted sinogram \(Hx+b\) and the observed data $y$, \(\mathcal{R}(x)\) is a regularization term enforcing the smoothness of the image, and $\mathcal{D}_0= \dom(g) = [\epsilon_0,+\infty)^N$ with $\epsilon_0\geq0$ incorporates the positivity constraint. We opt for a non-convex penalty function for $\mathcal{R}$, to enforce contrast in the restoration of PET piecewise-constant activity \cite{mmsubspaceL2L0}:
    \begin{equation}
        (\forall x \in \R^N) \quad \mathcal{R}(x) = \lambda \sum_{n=1}^N \vartheta \left(\|[\Delta x]_n\| \right)+\frac{\varepsilon}{2}\|x\|^2
    \label{eq:general_form_R}
    \end{equation}
    where $\vartheta : t \mapsto t^2/(2\delta^2+t^2)$ is the {Geman-McClure} potential, $\|\cdot\|$ denotes the Euclidean norm, $\Delta$ is the 2D discrete gradient operator such that for each pixel $n \in \llbracket 1,N\rrbracket$, $[\Delta x]_n=([\Delta_{\mathrm{h}} x]_n \; [\Delta_{\mathrm{v}} x]_n) \in \mathbb{R}^2$ (horizontal and vertical finite differences), and $(\lambda,\delta,\varepsilon) \in (0,+\infty)^3$.  
    
Function $F = f+g$ admits minimizers in $\mathcal{D}_0$ by Weierstrass' theorem as $f$ is lower-semicontinuous and coercive. 
{Let $\rho$ be defined as in  \cref{subsec:family-maj}. 
\cref{ass:A1:i} is met for any open set $\mathcal{D}=(-\tau, +\infty)^N$ such that $\tau \leq \rho$.
\cref{ass:A1:ii} is met as soon as $\mathcal{D}_0 \subset \mathcal{D}$, i.e. $\epsilon_0 > - \tau$,
which is obviously satisfied if $\tau > 0$.
Furthermore, function $F= f+g$ verifies
Assumption \ref{ass:A-KL} (K\L{} conditions). Indeed, the function $\mathcal{L}$ 
is definable in the o-minimal structure $\R_{\rm an,\exp}$. $\mathcal{R}$ is also lower-semicontinuous definable in this structure as the potential $\vartheta$ is analytic, and the quadratic term $\frac{\varepsilon}{2}\Vert x \Vert^2$ is semi-algebraic. Finally, $\mathcal{D}_0$ is characterized by $N$ affine inequalities, therefore is a semi-algebraic set.

In order to apply \cref{alg:VBMM}, and benefit from \cref{th:generalcv}, there remains to define suitable Bregman tangent majorants for $f$ that satisfy Assumptions \ref{ass:A2} and \ref{ass:A3}.
    
\subsection{Problem resolution}    
    We now  apply \cref{alg:VBMM} to perform PET reconstruction. For convenience, we focus on separable majorants associated with Legendre functions, for which the minimizer has a closed form expression. 
    Making use of the additivity property of the Bregman tangent majorants, we can majorize $\mathcal{L}$ and $\mathcal{R}$ separately. 

        \subsubsection{Majorizing \texorpdfstring{$\mathcal{R}$}{R}}
        \label{ssec:majR}
        Function $\mathcal{R}: \R^N \to \R$ \eqref{eq:general_form_R} is  differentiable on $\R^N$
        and $\nabla R$ is
        Lipschitzian on $\R^N$ with constant 
        $L_{\mathcal{R}}= \frac{\lambda}{\delta^2} \spnorm{\Delta}^2 + \varepsilon$, where $\spnorm{\cdot}$ is the operator norm
      (see  \cref{annexe:condition_lip}). 
      Hence, at any $z \in \R^N$, using the descent lemma, we have for any $M_\mathcal{R} \geq L_{\mathcal{R}}$,
       \begin{equation}
        (\forall x \in \R^N) \quad \mathcal{R}(x)\leq \mathcal{R}(z) + \langle \nabla \mathcal{R}(z),x-z\rangle + \frac{M_\mathcal{R}}{2}\Vert x-z \Vert^2 \triangleq Q_\mathcal{R}(x,z),
        \label{eq:descent_lemma}
        \end{equation}
        where it can be noticed that $Q_\mathcal{R}$ is a Bregman tangent majorant of $\mathcal{R}$ at $z$, associated with the Legendre function $\frac{M_\mathcal{R}}{2}\|\cdot\|^2$.

\subsubsection{Majorizing \texorpdfstring{$\mathcal{L}$}{L}} \label{sec:majL}
It can be noticed that
$
        (\forall x \in \mathcal{D})\quad \mathcal{L}(x) = \ell(x) + \sum_{m=1}^{M}H_m x
$
where function $\ell$ is defined 
in \eqref{def:notation_somme_log_hx_b}.
    By additivity, 
    we can derive a Bregman tangent majorant (\cref{def:BregmanMaj}) for $\mathcal{L}$ using the same Legendre function as for $\ell$:
\begin{align}
\label{eq:general_form_QL_maj}
    (\forall x \in \mathcal{D}) \quad
    Q_\mathcal{L}(x,z) &= \mathcal{L}(z) + \langle \nabla \mathcal{L}(z), x-z \rangle + D_{h_z}(x,z) \geq \mathcal{L}(x).
\end{align}
\reviewchanges{Function $\mathcal{L}$ is $L_\mathcal{L}$-Lipschitz differentiable with
\begin{equation}
    L_\mathcal{L} = \spnorm{H}^2 \underset{1 \leq m \leq M}{\text{max}} \frac{y_m}{b_m^2}\label{eq:LipLoss}
\end{equation}
(see \cref{annexe:condition_lip} for detailed calculations), so a simple choice would be to build $Q_\mathcal{L}$ as a quadratic tangent majorant with constant curvature, associated with the Legendre function $\frac{L_\mathcal{L}}{2}\| \cdot \|^2$. This choice is later referred to as \textbf{Maj}-Lip. Several other constructions of separable Bregman tangent majorants for our data fidelity term $\mathcal{L}$ can be obtained, leveraging the majorization results for $\ell$ from \cref{subsec:family-maj}. Among those, we will compare (Log-shift)~\textbf{Maj-\majO} \eqref{eq:h0}, \textbf{Maj-\majII}~\eqref{eq:a2}, and \textbf{Maj-\majVIII} \eqref{eq:h8} with $\mu=\rho$ 
and $\mathcal{D} = (-\rho,+\infty)^N$, (Log-0) \textbf{Maj}-\majI~\eqref{eq:h1}, and \textbf{Maj}-\majVI~\eqref{eq:majh6}
with $\mathcal{D} = (0,+\infty)^N$, and (Quadratic) \textbf{Maj}-\majVp~\eqref{eq:h5p} and \textbf{Maj}-Lip~\eqref{eq:LipLoss},~
with $\mathcal{D} = (-\rho,\infty)^N$. 
The rationale for our selection is the following. Within each family, we selected the Legendre functions with minimal computational cost, and best approximation accuracy (relying on the order relations  between majorants established in \cref{subsec:family-maj}). Namely,  \textbf{Maj-\majV} $\,$ was discarded, as it requires $M$ backprojections at each iteration, as opposed to \textbf{Maj}-\majVp~(see \cref{tab:operations} in \cref{annex:nbproj}). \textbf{Maj}-\majIII~was also discarded as it 
provides less accurate
majorizations than \textbf{Maj}-\majVIII~(see \cref{prop:listmajviii}). Though theoretically less accurate than \textbf{Maj}-\majVIII~(resp. \textbf{Maj}-\majVI), we kept \textbf{Maj}-\majO~(resp. \textbf{Maj}-\majI) in our benchmarks because these are used in state-of-the-art Poisson image reconstruction techniques \cite{FesslerHero1995}. Our choices also intend to perform an ablation study, namely constant metrics \textbf{Maj}-\majII~and \textbf{Maj}-Lip~versus varying metrics; quadratic metrics \textbf{Maj}-\majVp~\eqref{eq:h5p} and \textbf{Maj}-Lip~versus logarithmic ones. Finally, we emphasize that, when using the quadratic metrics, our algorithm reduces to a projected gradient method (with \textbf{Maj}-Lip~\eqref{eq:LipLoss}), and a variable metric forward-backward scheme \cite{ChouzenouxPesquetRepetti2015,BONETTINI2021113192} (with \textbf{Maj}-\majVp~\eqref{eq:h5p}), while the constant logarithmic metric \textbf{Maj}-\majII~amounts to implementing the Bregman proximal gradient algorithm~\cite{bauschke_descent_2017}. For the aforementioned choices of majorant functions and sets $\mathcal{D}$, \cref{ass:A2} is satisfied on a domain compatible with \cref{ass:A1:i}.}

\subsubsection{Final algorithm}

We now deduce Bregman tangent majorants of $f=\mathcal{L}+\mathcal{R}$ combining \eqref{eq:descent_lemma} and \eqref{eq:general_form_QL_maj}. For all $z \in \mathcal{D}$,
\begin{align}
    (\forall x \in \mathcal{D}) \quad Q_f(x,z) &= Q_\mathcal{L}(x,z) + Q_\mathcal{R}(x,z) \\
    &= f(z) + \langle \nabla f(z),x-z\rangle + D_{h_z}(x,z)+ \frac{M_\mathcal{R}}{2}\Vert x-z \Vert^2 
    \label{eq:Qf_withdiv}
\end{align}
with $(h_z,\mathcal{D})$ one of the seven settings 
mentioned earlier. \reviewchanges{It may be noticed that $Q_f$ is a Bregman tangent majorant of $f$ at $z$ associated with the Legendre function $x\mapsto h_z(x) + \frac{M_\mathcal{R}}{2}\|x\|^2$, where $h_z$ depends on our loss majorant selection in \cref{sec:majL}.}
By \eqref{eq:general_form_QL_maj}, we have for every $k \in \mathbb{N}$, $Q_f(x,x^k)-f(x) \geq Q_{\mathcal{R}}(x,x^k)-\mathcal{R}(x)$. Moreover, by the Lipschitz property of $\nabla \mathcal{R}$ and majorization ~\eqref{eq:descent_lemma}, for every $k \in \mathbb{N}$, $Q_{\mathcal{R}}(x,x^k)-\mathcal{R}(x) \geq \frac{1}{2}(M_{\mathcal{R}}-L_{\mathcal{R}}) \|x - x^k\|^2$. Therefore, it is sufficient to set $M_{\mathcal{R}} > L_{\mathcal{R}}$ for \cref{ass:A3:i} to be satisfied 
 with $\underline{\gamma} = M_{\mathcal{R}}-L_{\mathcal{R}}>0$. 
Moreover, as shown in \cref{annexe:condition_lip}, {$\nabla f$ is Lipschitz continuous on $[0,+\infty)^N$, and $-\nabla h_z$ is Lipschitz continuous on any compact subset of $\mathcal{D}_0$, hence \cref{ass:A3:ii} holds.

Let $x^0 \in \mathcal{D}_0 \subset \mathcal{D}$. \Cref{alg:VBMM} then reads
\begin{align}
    (\forall k\in\mathbb{N}) \quad x^{k+1} &= \operatorname*{argmin}_{x \in \mathbb{R}^N} \left(Q_f(x,x^k) + \iota_{\mathcal{D}_0}(x)\right) = \operatorname*{argmin}_{x\in \mathcal{D}_0} \left(Q_f(x,x^k)\right).\end{align}
Our choices of $h_z$ are separable, hence so is the majorant function, i.e., we can rewrite, for every $z \in \mathcal{D}_0$,
$Q_f(\cdot,z)$ in the form: $Q_f(x,z) = \sum_{n=1}^N q_{n}(x_n,z)$.
Hence, since $q(\cdot,z)$ is (strictly) convex and $\mathcal{D}_0 = [\epsilon_0,+\infty)^N$, we have
    \begin{align}
    (\forall k\in\mathbb{N}) \quad x^{k+1} & = \left(\operatorname*{argmin}_{\xi \in [\epsilon_0,+\infty)} q_n(\xi,x^k)\right)_{1 \leq n \leq N} =\left(\operatorname{proj}_{[\epsilon_0,+\infty)} (u_n(x^k))\right)_{1 \leq n \leq N},
    \label{eq:VBMMM_update}
\end{align}
where $\operatorname{proj}_{[\epsilon_0,+\infty)}$ is the projection onto  $[\epsilon_0,+\infty)$, and $u_n\colon \mathcal{D} \to \mathbb{R}\colon z \mapsto 
\operatorname*{argmin}_{\xi \in \mathcal{D}}
q(\xi,z)$.

\cref{table:updates} recaps the expression for the validity domains $\mathcal{D}$, $\mathcal{D}_0$, and the update formulas for $u_n(z)$, for each of the majorants chosen for $\mathcal{L}$. Detailed calculations are provided in \cref{annexe:details_updates}.\\


       \begin{table}[htb!]
            \centering
            \footnotesize
            \renewcommand{\arraystretch}{2.5} 
            \resizebox{0.9\linewidth}{!}{%
            \begin{tabular}{|c|c|c|c|c|c|c|}
            \hline
                Family & \textbf{Maj} & $\mathcal{D}$ & {$\epsilon_0$} & \textbf{Update} components $u_n(z)$ & $d_n(z)$ & $a_n(z)$ \\
                \hline 
                \hline 
                \multirow{3}{*}{Log-shift} & \textbf{Maj-\majO} & \multirow{3}{*}{$(-\rho,+\infty)^N$} & \multirow{3}{*}{0} & \multirow{3}{*}{$\displaystyle\frac{\sqrt{(d_n(z)- M_\mathcal{R} \rho)^2+4 M_\mathcal{R} a_{n}(z)}-d_n(z)-M_\mathcal{R} \rho}{2 M_\mathcal{R} }$} & \multirow{3}{*}{$[\nabla f(z)]_n + \frac{a_{n}(z)}{z_n+\rho} - M_\mathcal{R} z_n$} & \eqref{eq:a0} \\
                 & \reviewchanges{\textbf{Maj-\majII}} &  &  &  &  & \eqref{eq:a2} \\
                 & \textbf{Maj-\majVIII} &  &  &  &  & \eqref{eq:a8} \\
                \hline
                 \multirow{2}{*}{Log-0}& \textbf{Maj-\majI} &  \multirow{2}{*}{$(0,+\infty)^N$} &  \multirow{2}{*}{$0.01$}& \multirow{2}{*}{$\displaystyle\frac{\sqrt{d_n(z)^2+4 M_\mathcal{R} {a_{n}}(z)}-d_n(z)}{2 M_\mathcal{R} }$} & \multirow{2}{*}{$[\nabla f(z)]_n + \frac{a_n(z)}{z_n} - M_\mathcal{R} z_n$} & \eqref{eq:a0} \\
                & \textbf{Maj-\majVI} & &  & & & \eqref{eq:maj6a6}\\
                \hline
                \multirow{2}{*}{Quadratic}& \textbf{Maj-\majVp} & \multirow{2}{*}{$(-\rho,+\infty)^N$} & \multirow{2}{*}{{$0$}} & \multirow{2}{*}{$z_n - \displaystyle\frac{[\nabla f(z)]_n}{a_{n}(z)+M_\mathcal{R}}$} & $-$ & \eqref{eq:a5p}\\                
                & \reviewchanges{\textbf{Maj-Lip}} &  &  &  & $-$ & \reviewchanges{\eqref{eq:LipLoss}} \\
                \hline
            \end{tabular}
            }
            \caption{Compared majorant Legendre functions, with domain $\mathcal{D}$. The proposed Bregman MM update takes the form  \eqref{eq:VBMMM_update}. We provide our choice of $\epsilon_0$ and the expression for all $z = (z_n)_{1 \leq n \leq N}  \in \mathcal{D}$ and $n \in \llbracket 1,N\rrbracket$, of $u_n(z)$ as a function of coefficients $a_n(z)$ and $d_n(z)$. $M_\mathcal{R}$ should be chosen greater than $L_{\mathcal{R}}$. \reviewchanges{Metrics \textbf{Maj-\majII} and \textbf{Maj-Lip} are constant w.r.t. $z$.}
            }
            \label{table:updates}
        \end{table}
        
    \subsection{Numerical results}

Simulations were performed in 2D on 2 slices of a brain phantom\reviewchanges{, denoted slice A and slice B,} of size 128$\times 128$ (pixel size $2.03$mm). \reviewchanges{The phantom was} derived from real $^{18}$F-FDG PET data from a healthy volunteer (with PET signal measured on 100 \reviewchanges{anatomical} regions of interest), \reviewchanges{following the approach of  \cite{analytical_simulator} (see \cref{annex:simulation} for details)}. The reconstruction framework was developed 
using a custom CUDA-accelerated forward and backward projector with configurable block sizes \cite{Rossignol2022}.
The support of $\overline{x}$ is chosen so that the reconstructed pixels belong to a field of view described by a binarized morphological closing of the attenuation image of the phantom (always available in PET imaging). For the quadratic majorant \textbf{Maj-\majVp}, a Taylor expansion was used to evaluate the curvature \eqref{eq:a5p} for components $z_n$ near zero, leading to improved numerical stability. All experiments were performed on a NVIDIA A100-SXM4 with 80 GB of memory. \\

\reviewchanges{Two convergence criteria were considered, namely the iterate relative distance to the asymptotic image $x^\infty$ (obtained in practice by reconstructing with a very large number of iterations) and an optimality criterion, measuring first-order stationarity via the infinity norm of the unit-step projected gradient residual:
\begin{equation}
(\forall k \in \mathbb{N}) \quad G_k := x^k-P_{\mathcal D_0}\bigl(x^k-\nabla f(x^k) \bigr).  
\end{equation}
We also evaluate the computational time, as well as the image quality metrics comparing the restored image (for a given stopping criterion) and the ground truth $\overline{x}$, namely the perceptual index SSIM, NRMSE, PSNR, and CNR. Calculation details for the image quality metrics used are presented in \cref{annex:metrics}. Finally, for comparison purpose, we also display the results of the non penalized ML-EM algorithm \cite{DePierro1993}, obtained by setting $\mathcal{R} \equiv 0$, and $\mathcal{D}_0 = [0,+\infty)^N$ in our framework.
}

\reviewchanges{In \Cref{tab:table_15s,tab:table_120s,tab:kkt_0p001}, we display  the results for the compared algorithms, for three different stopping criteria namely a fixed computational time of 15 and 120 seconds, respectively, and the projected gradient residual norm satisfying $\|G_k\|_\infty \leq 10^{-3}$. We display a $``/"$ symbol for algorithms that do not meet the stopping criteria after $10^6$ iterations. Best results are highlighted in bold. \Cref{fig:exreconstructions_slice71,fig:exreconstructions_slice65} each illustrate\reviewchanges{, for a given phantom slice,} examples of reconstructions obtained with our proposed algorithm, and \reviewchanges{five} different majorants. Finally, \cref{fig:cvprofile} displays convergence profiles obtained for the compared majorants.}


\begin{table}[htb!]
    \centering
    {\scriptsize
    \begin{tabular}{lccccccc}
\toprule
          Majorant &  Nb. it. &  $\|G_k\|_\infty$ &  $\frac{\|x^k-x^\infty\|}{\|x^\infty\|}$ &  NRMSE &   SSIM &    PSNR &    CNR \\
\toprule
            ML-EM &            648 &                     0.0062 &                           0.650 &  0.431 & 0.892 & 16.61 & 1.76 \\
\midrule
  Log-shift Maj-1 &            373 &                     0.0060 &                           0.187 &  0.404 & 0.895 & 17.17 & 1.86 \\
  Log-shift Maj-2 &            645 &                     1.0210 &                           0.465 &  0.711 & 0.631 & 12.26 & 1.24 \\
  Log-shift Maj-4 &            650 &                     \textbf{0.0018} &                           \textbf{0.154} &  0.401 & 0.894 & 17.24 & 1.88 \\
      Log-0 Maj-5 &            374 &                     0.0190 &                           0.215 &  0.409 & 0.893 & 17.06 & 1.85 \\
      Log-0 Maj-6 &            648 &                     0.0064 &                           0.157 &  \textbf{0.400} & \textbf{0.896} & \textbf{17.25} & \textbf{1.89} \\
  Quadratic Maj-8 &            372 &                     0.0137 &                           0.241 &  0.440 & 0.863 & 16.42 & 1.69 \\
Quadratic Maj-Lip &            663 &                     0.0824 &                           0.337 &  0.549 & 0.756 & 14.50 & 1.26 \\
\bottomrule
\end{tabular}
    }
    \caption{Comparison of convergence criteria and image quality metrics after iterating reconstruction for 15 seconds (phantom slice A).}
    \label{tab:table_15s}
\end{table}

 \begin{table}[htb!]
     \centering
     {\scriptsize
     \begin{tabular}{lccccccc}
\toprule
          Majorant &  Nb. it. &  $\|G_k\|_\infty$ &  $\frac{\|x^k-x^\infty\|}{\|x^\infty\|}$ &  NRMSE &   SSIM &    PSNR &    CNR \\
\toprule
            ML-EM &           5182 &                     0.0046 &                           0.356 &  0.827 & 0.803 & 10.94 & 0.83 \\
\midrule
  Log-shift Maj-1 &           2996 &                     0.0009 &                           0.088 &  0.419 & 0.899 & 16.85 & 1.85 \\
  Log-shift Maj-2 &           5094 &                     0.2622 &                           0.291 &  0.558 & 0.735 & 14.35 & 1.39 \\
  Log-shift Maj-4 &           5120 &                     \textbf{0.0001} &                           0.064 &  0.435 & 0.900 & 16.53 & 1.84 \\
      Log-0 Maj-5 &           2993 &                     0.0072 &                           0.111 &  0.423 & \textbf{0.901} & 16.77 & 1.86 \\
      Log-0 Maj-6 &           5141 &                     0.0022 &                           \textbf{0.062} &  0.433 & \textbf{0.901} & 16.55 & 1.84 \\
  Quadratic Maj-8 &           2988 &                     0.0017 &                           0.127 &  \textbf{0.401} & 0.892 & \textbf{17.23} & \textbf{1.88} \\
Quadratic Maj-Lip &           5263 &                     0.0088 &                           0.195 &  0.444 & 0.857 & 16.33 & 1.64 \\
\bottomrule
\end{tabular}
     }
     \caption{Comparison of convergence criteria and image quality metrics after iterating reconstruction for 120 seconds (phantom slice A).}
     \label{tab:table_120s}
 \end{table}

\begin{table}[htb!]
    \centering
    {\scriptsize
    \begin{tabular}{lccccccc}
\toprule
          Majorant & Time (s) & Nb. it. &  $\frac{\|x^k-x^\infty\|}{\|x^\infty\|}$ &  NRMSE &   SSIM &    PSNR &    CNR \\
\toprule
              ML-EM &     / &             / &                              / &     / &     / &     / &    / \\
\midrule
  Log-shift Maj-1 & 107.9 &          2693 &                          0.093 & 0.416 & 0.898 & 16.90 & 1.86 \\
  Log-shift Maj-2 &     / &             / &                              / &     / &     / &     / &    / \\
  Log-shift Maj-4 &  \textbf{26.8} &          \textbf{1158} &                          0.131 & 0.405 & 0.892 & 17.15 & \textbf{1.87} \\
      Log-0 Maj-5 &     / &             / &                              / &     / &     / &     / &    / \\
      Log-0 Maj-6 & 271.9 &         11437 &                          \textbf{0.033} & 0.444 & \textbf{0.899} & 16.34 & 1.83 \\
  Quadratic Maj-8 & 190.9 &          4734 &                          0.108 & \textbf{0.403} & 0.890 & \textbf{17.18} & \textbf{1.87} \\
Quadratic Maj-Lip & 787.1 &         33781 &                          0.059 & 0.406 & 0.881 & 17.13 & 1.86 \\
\bottomrule
\end{tabular}
    }
    \caption{Comparison of metrics after iterating reconstruction until reaching 
$\|G_k\|_\infty\leq 10^{-3}$ (phantom slice A).}
    \label{tab:kkt_0p001}
\end{table}

\newcommand{\reconslinewidthratio}{0.17}
\newcommand{\subfighspace}{\hspace{0.1cm}}
\newcommand{\hspacewcbar}{\hspace{0.4cm}}

\begin{figure}[htb!]
    \centering
    \hspace{0.75cm}
    \begin{minipage}[c]{0.15\linewidth}
        \centering
        \subfloat{\label{fig:recons_phantom}%
            \stackunder{\includegraphics[width=2.5cm]{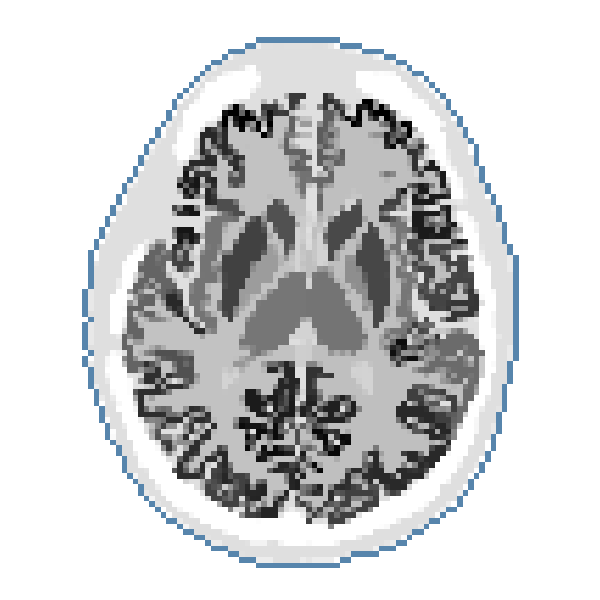}}%
            {\shortstack{{\scriptsize Reference} \\ {\tiny $ $}}}}%
    \end{minipage}%
    \hfill
    \begin{minipage}[c]{0.55\linewidth}
        \centering
        \begin{minipage}[b]{0.3\linewidth}
            \centering
            \subfloat{\label{fig:recons_mm_logshift}%
                \stackunder{\includegraphics[width=2.5cm]{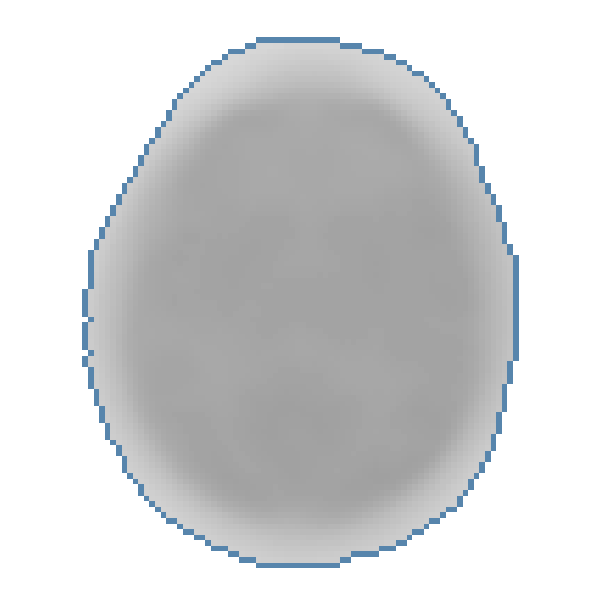}}%
                {\shortstack{{\scriptsize Log-shift \textbf{Maj-}\majII} \\ {\tiny NRMSE:0.711} \\ {\tiny SSIM:0.631}}}}%
        \end{minipage}%
        \hfill
        \begin{minipage}[b]{0.3\linewidth}
            \centering
            \subfloat{\label{fig:recons_mm_logshift2_bpg}%
                \stackunder{\includegraphics[width=2.5cm]{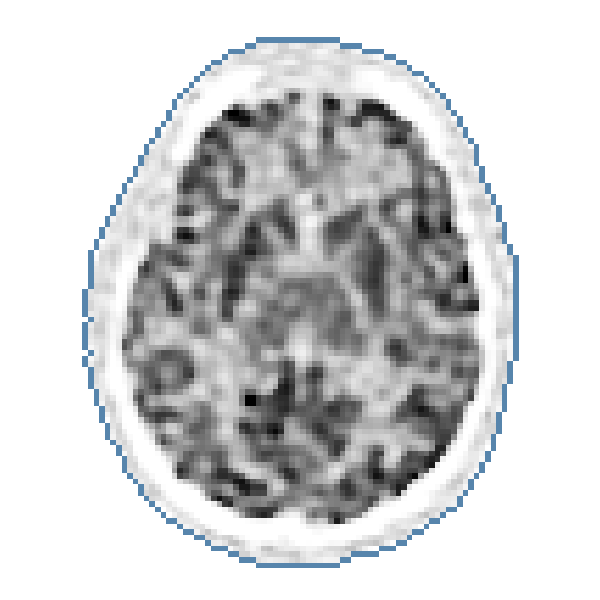}}%
                {\shortstack{{\scriptsize Log-shift \textbf{Maj-}\majVIII} \\ {\tiny NRMSE:0.401} \\ {\tiny SSIM:0.894}}}}%
        \end{minipage}%
        \hfill
        \begin{minipage}[b]{0.3\linewidth}
            \centering
            \subfloat{\label{fig:recons_mm_log0}%
                \stackunder{\includegraphics[width=2.5cm]{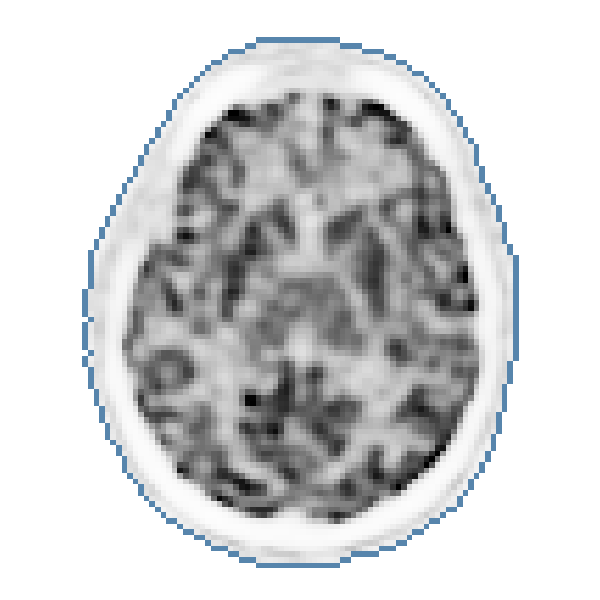}}%
                {\shortstack{{\scriptsize Log-0 \textbf{Maj-}\majI} \\ {\tiny NRMSE:0.409} \\ {\tiny SSIM:0.893}}}}%
        \end{minipage}%

        \vspace{0.25cm}
        \begin{minipage}[b]{0.3\linewidth}
            \centering
            \subfloat{\label{fig:recons_em}%
                \stackunder{\includegraphics[width=2.5cm]{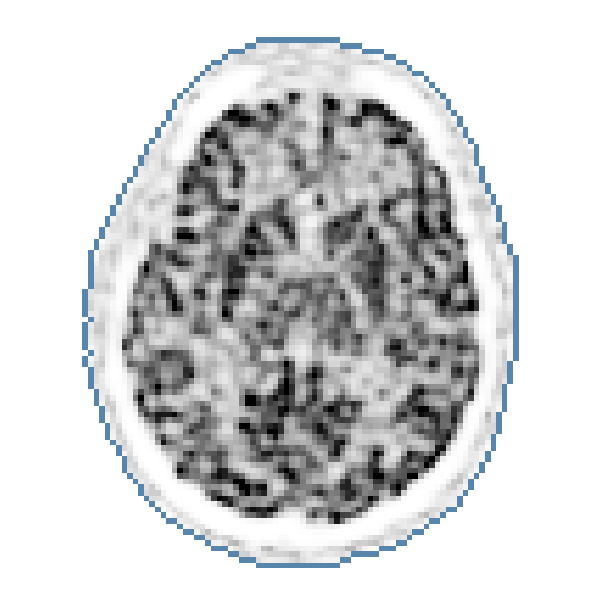}}%
                {\shortstack{{\scriptsize ML-EM} \\ {\tiny NRMSE:0.430} \\ {\tiny SSIM:0.893}}}}%
        \end{minipage}%
        \hfill
        \begin{minipage}[b]{0.3\linewidth}
            \centering
            \subfloat{\label{fig:recons_mm_quad}%
                \stackunder{\includegraphics[width=2.5cm]{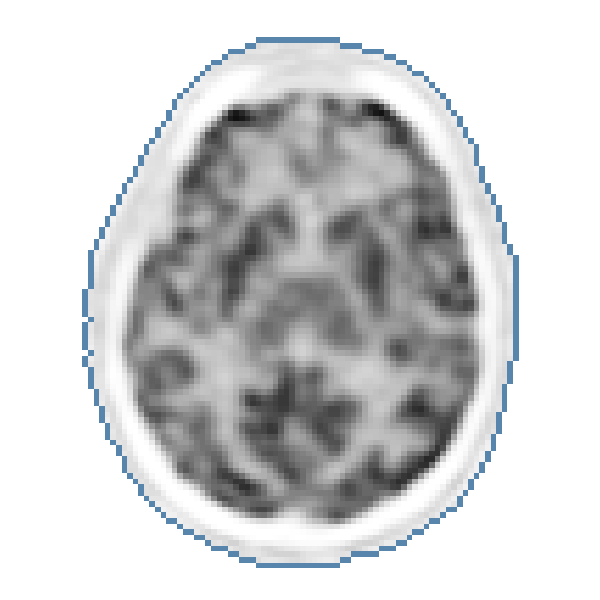}}%
                {\shortstack{{\scriptsize Quadratic \textbf{Maj-}\majVp} \\ {\tiny NRMSE:0.441} \\ {\tiny SSIM:0.863}}}}%
        \end{minipage}%
        \hfill
        \begin{minipage}[b]{0.3\linewidth}
            \centering
            \subfloat{\label{fig:recons_mm_projgrad}%
                \stackunder{\includegraphics[width=2.5cm]{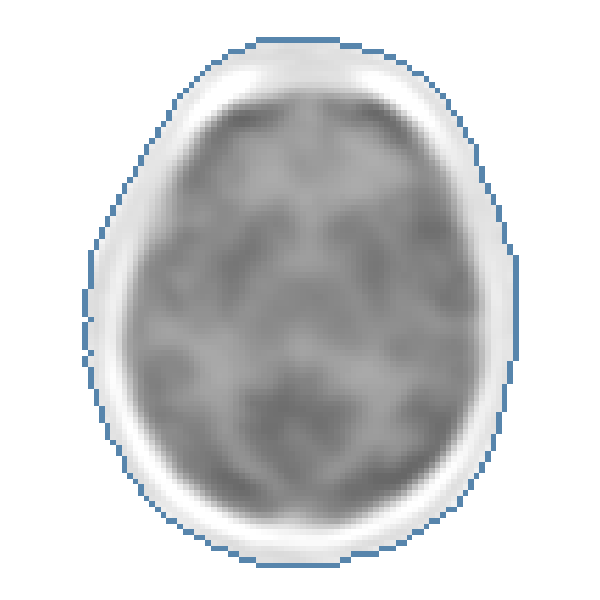}}%
                {\shortstack{{\scriptsize Quadratic \textbf{Maj-}Lip} \\ {\tiny NRMSE:0.550} \\ {\tiny SSIM:0.754}}}}%
        \end{minipage}%

    \end{minipage}%
    \hfill
    \begin{minipage}[c]{0.035\linewidth}
        \centering
        \includegraphics[width=0.55cm]{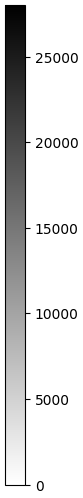}%
    \end{minipage}%
    \hspace{0.75cm}
    \caption{Example of reconstructions of the phantom slice A (left) after iterating for 15 seconds with 6 different methods. Blue lines show the mask contour.}
    \label{fig:exreconstructions_slice71}
\end{figure}
\begin{figure}[htb!]
    \centering
    \hspace{0.75cm}
    \begin{minipage}[c]{0.15\linewidth}
        \centering
        \subfloat{\label{fig:recons_phantom}%
            \stackunder{\includegraphics[width=2.5cm]{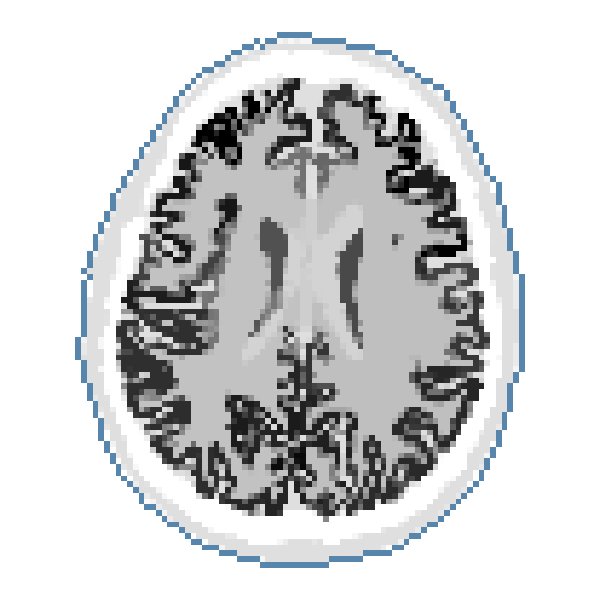}}%
            {\shortstack{{\scriptsize Reference} \\ {\tiny $ $}}}}%
    \end{minipage}%
    \hfill
    \begin{minipage}[c]{0.55\linewidth}
        \centering
        \begin{minipage}[b]{0.3\linewidth}
            \centering
            \subfloat{\label{fig:recons_mm_logshift}%
                \stackunder{\includegraphics[width=2.5cm]{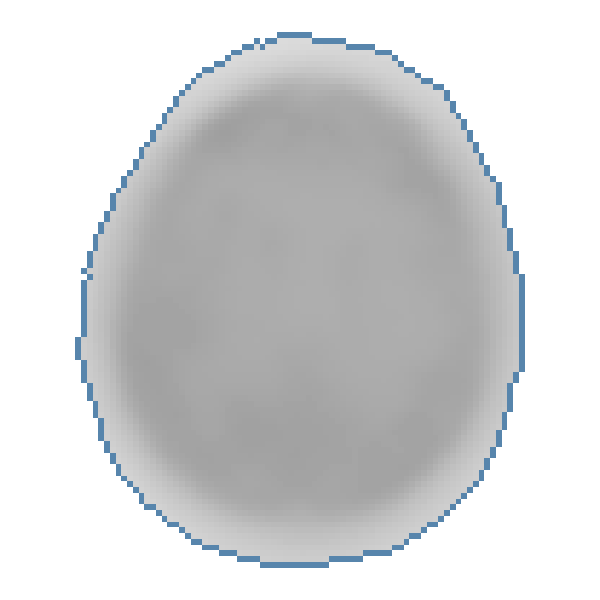}}%
                {\shortstack{{\scriptsize Log-shift \textbf{Maj-}\majII} \\ {\tiny NRMSE:0.792} \\ {\tiny SSIM:0.647}}}}%
        \end{minipage}%
        \hfill
        \begin{minipage}[b]{0.3\linewidth}
            \centering
            \subfloat{\label{fig:recons_mm_logshift2_bpg}%
                \stackunder{\includegraphics[width=2.5cm]{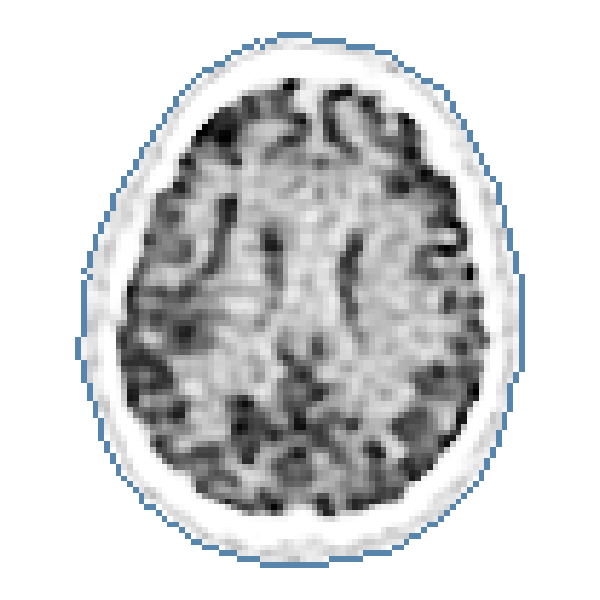}}%
                {\shortstack{{\scriptsize Log-shift \textbf{Maj-}\majVIII} \\ {\tiny NRMSE:0.443} \\ {\tiny SSIM:0.896}}}}%
        \end{minipage}%
        \hfill
        \begin{minipage}[b]{0.3\linewidth}
            \centering
            \subfloat{\label{fig:recons_mm_log0}%
                \stackunder{\includegraphics[width=2.5cm]{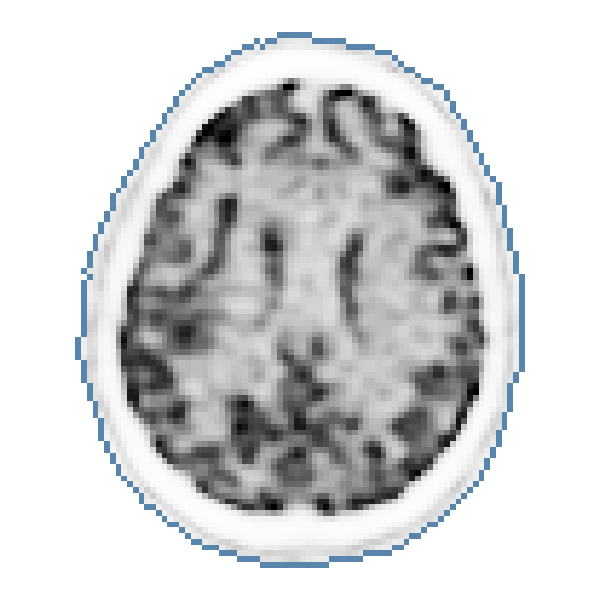}}%
                {\shortstack{{\scriptsize Log-0 \textbf{Maj-}\majI} \\ {\tiny NRMSE:0.454} \\ {\tiny SSIM:0.895}}}}%
        \end{minipage}%

        \vspace{0.25cm}
        \begin{minipage}[b]{0.3\linewidth}
            \centering
            \subfloat{\label{fig:recons_em}%
                \stackunder{\includegraphics[width=2.5cm]{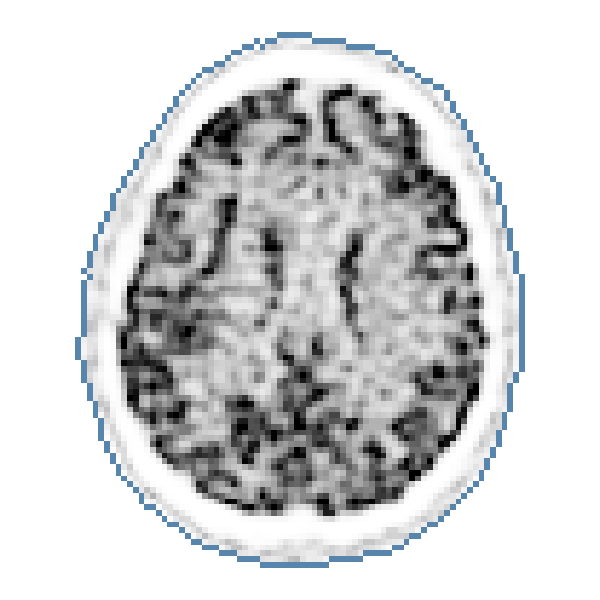}}%
                {\shortstack{{\scriptsize ML-EM} \\ {\tiny NRMSE:0.464} \\ {\tiny SSIM:0.897}}}}%
        \end{minipage}%
        \hfill
        \begin{minipage}[b]{0.3\linewidth}
            \centering
            \subfloat{\label{fig:recons_mm_quad}%
                \stackunder{\includegraphics[width=2.5cm]{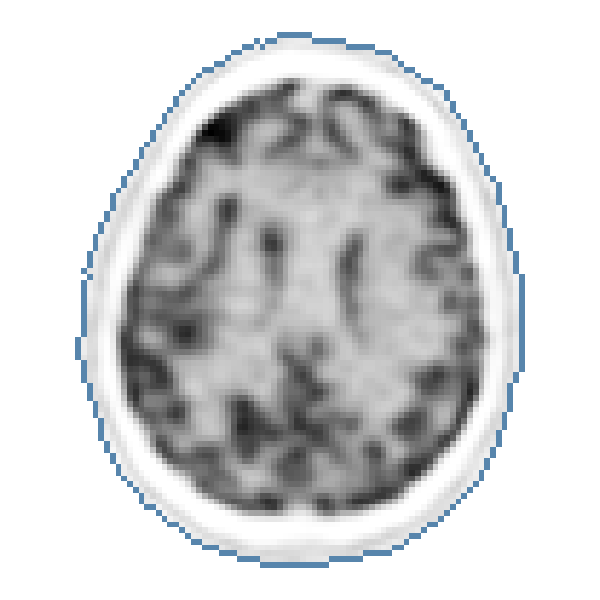}}%
                {\shortstack{{\scriptsize Quadratic \textbf{Maj-}\majVp} \\ {\tiny NRMSE:0.487} \\ {\tiny SSIM:0.871}}}}%
        \end{minipage}%
        \hfill
        \begin{minipage}[b]{0.3\linewidth}
            \centering
            \subfloat{\label{fig:recons_mm_projgrad}%
                \stackunder{\includegraphics[width=2.5cm]{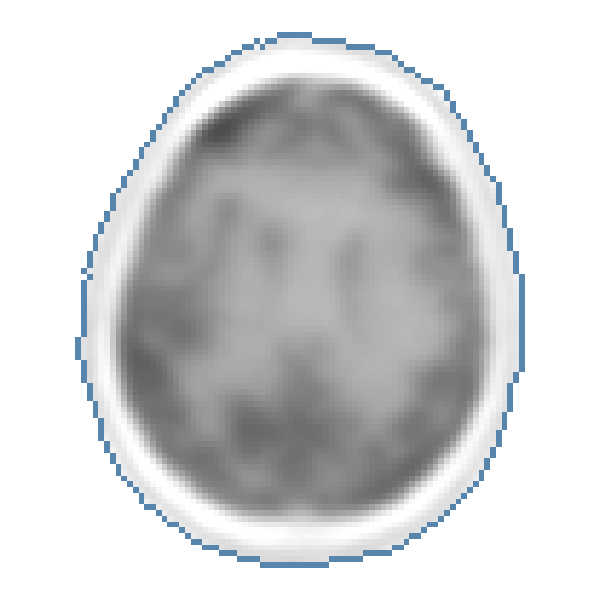}}%
                {\shortstack{{\scriptsize Quadratic \textbf{Maj-}Lip} \\ {\tiny NRMSE:0.615} \\ {\tiny SSIM:0.768}}}}%
        \end{minipage}%
    \end{minipage}%
    \hfill
    \begin{minipage}[c]{0.035\linewidth}
        \centering
        \includegraphics[width=0.55cm]{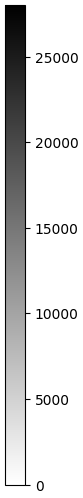}%
    \end{minipage}%
    \hspace{0.75cm}
    \caption{Example of reconstructions of the phantom slice B (left) after iterating for 15 seconds with 6 different methods. Blue lines show the mask contour.}
    \label{fig:exreconstructions_slice65}
\end{figure}


\begin{figure}[htb!]
\centering
    \subfloat{%
        \stackunder{\includegraphics[width=0.48\linewidth]{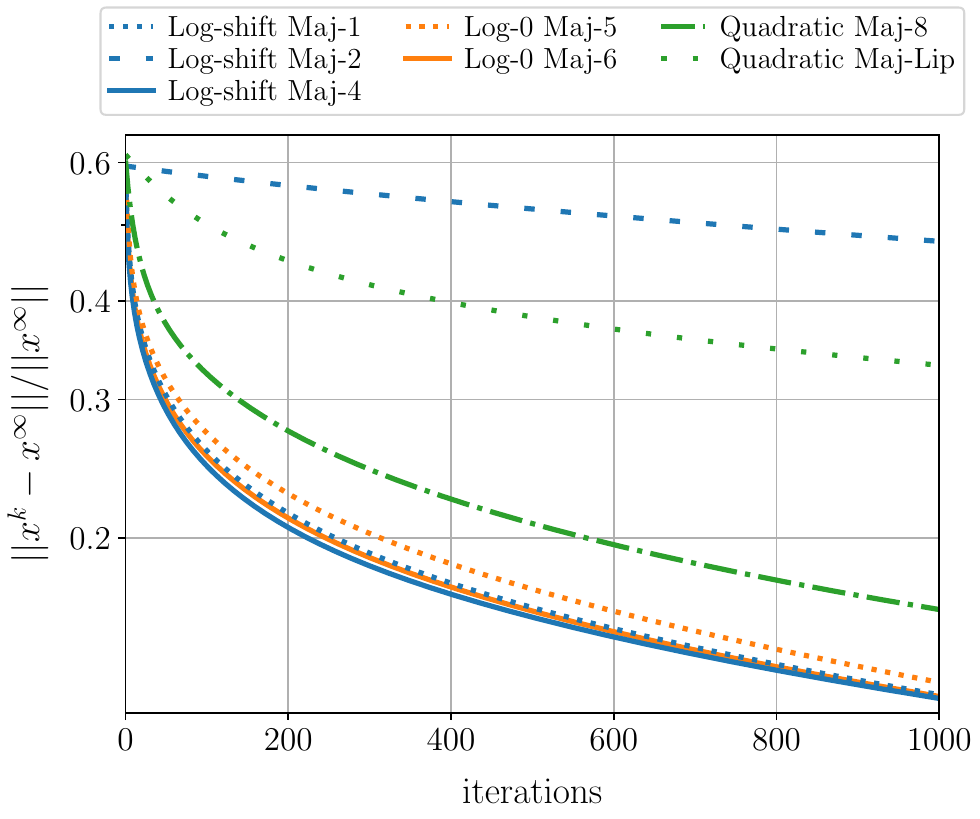}}%
        {{\footnotesize (a)}}}%
    \hfill
    \subfloat{%
        \stackunder{\includegraphics[width=0.48\linewidth]{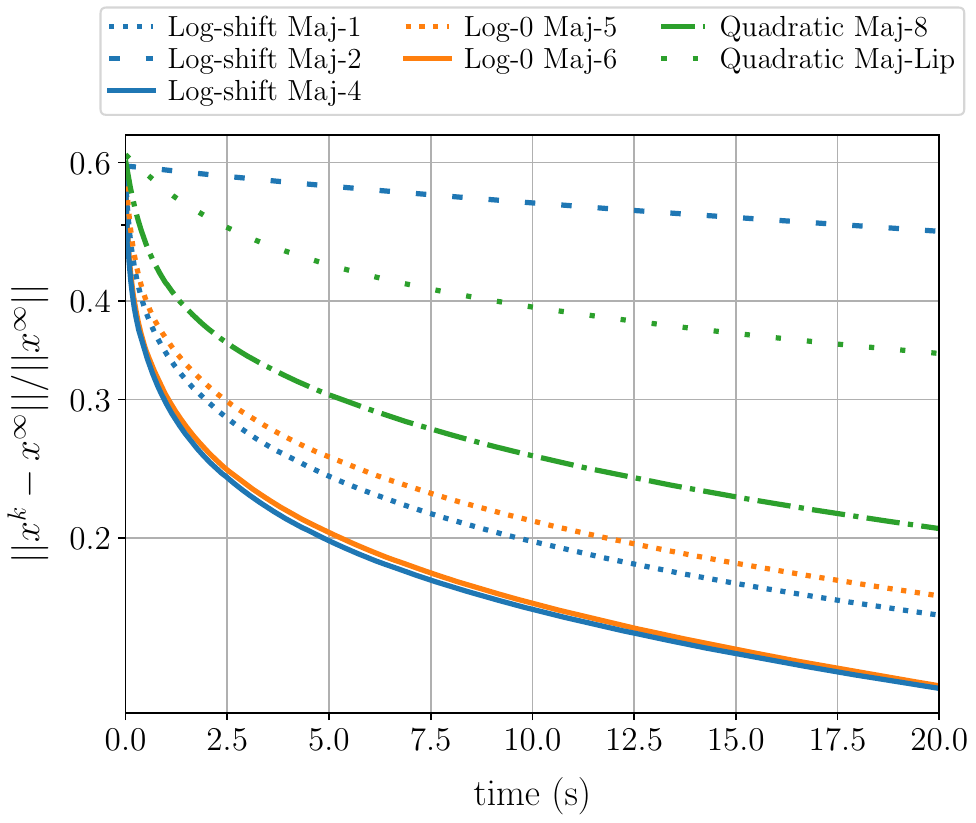}}%
        {{\footnotesize (b)}}}%
\caption{Iterates relative distance over iterations (a) and computational time (b) using the different majorants and our nonconvex Geman-McClure regularization. An approximation of the limit point $x^{\infty}$ was obtained by running each algorithm ten times the number of iterations required to meet the stopping criterion.}
\label{fig:cvprofile}
\end{figure}


\reviewchanges{
For a given computational budget, the variable metric log-shift majorant \textbf{Maj-\majVIII} yields the best reconstructions: higher scores reflect the visible improvements in detail preservation and contrast within low-dose regions. The other majorants within log-shift family have lower performance, in particular  \textbf{Maj-\majII} (instance of BPG) exhibits very slow convergence, probably due to its non-adaptive metric. Images obtained with \textbf{Maj-}Lip (projected gradient) and \textbf{Maj-}$\majII$, which are majorants with constant metrics, barely show the brain features after 15 seconds as the algorithms have very slow convergence on these examples.}    

\reviewchanges{
As shown in
\cref{fig:cvprofile}, quadratic majorants converge significantly slower than those \reviewchanges{with variable metric} from the logarithmic families, which explains the lower scores also visible in the reconstructions displayed in \cref{fig:exreconstructions_slice71,fig:exreconstructions_slice65} obtained after 15 seconds. As expected, tighter iterate-dependent majorants yield faster convergence, hence a better image quality within a shorter computational time, consistently with the tightness order relations between majorants (summarized in \cref{fig:orderrelations}). The differences observed between \cref{fig:cvprofile}(a) and \cref{fig:cvprofile}(b) stem from the fact that majorants \textbf{Maj-\majVIII} and \textbf{Maj-\majVI} require only a single backprojection per iteration instead of two (see \cref{annex:nbproj}), the latter being the most time-consuming step of the algorithm in the tomographic reconstruction setting. The ML-EM relies on an unregularized loss, hence resulting in noisy textures, as confirmed by the poorer quality metrics. To summarize, the experimental study illustrates the essential role of (i) variable metrics, (ii) non-Euclidean Bregman metrics, and (iii) tight majorant constructions for fast and accurate Poisson image reconstruction.}

\section{Conclusion}
In this work, we established new theoretical guarantees for the variable Bregman majorization–minimization algorithm. The proposed proof 
is grounded on general
assumptions which apply to a wide class
of nonsmooth nonconvex objectives. Iterates convergence to a critical point was established under the Kurdyka–\L{}ojasiewicz property. We illustrated the versatility of the proposed framework by deriving constructive rules for Bregman majorant design, and we derived a broad class of separable majorants for Poisson data fidelity term. \reviewchanges{Finally, numerical experiments on a challenging nonconvex Poisson image reconstruction problem demonstrated the practical effectiveness of the proposed framework, highlighting the impact of carefully designed Bregman majorants on both convergence speed and reconstruction quality}. The proposed approach should be applicable beyond PET and allow the development of data-driven strategies grounded on its use. \reviewchanges{Extension of the convergence analysis to stochastic settings \cite{Kereta2021,Twyman2023,Ehrhardt2025}, and ordered subset implementations \cite{Ahn2003}, will be considered as future research avenues.}

\section*{Acknowledgments}
The authors would like to thank Claire Rossignol for her contribution to the early development of this project.

\bibliographystyle{siamplain}
{\small
\bibliography{references}
}



\clearpage
\begin{appendices}

\makeatletter
\let\oldaddcontentsline\addcontentsline
\renewcommand{\addcontentsline}[3]{}
\makeatother

\section{Proof of \cref{prop:log0majorants}} 
\label{annex:prooflog0}
Let $\mathcal{D} = (0,+\infty)^N$ and $z \in \mathcal{D}$. 
    \begin{description}
    \item[(\textbf{Maj-$\majI$})] 
        For every $x\in \mathcal{D}$,
            \begin{equation*}
                (x-z)^\top C_{h_{\majO,z}}(x-z)
                \le \sum_{n=1}^N a_{\majO,n}(z) \int_{0}^1 \frac{(1-t)\, (x_{n}-z_{n})^2}{((1-t) z_n+t x_{n})^2}\,dt
                = 
                (x-z)^\top C_{h_{\majI,z}}(x-z).
            \end{equation*}
            Hence $h_{\majO,z}\preceq h_{\majI,z}$, proving \ref{prop:log0majorants}\ref{prop:ordre_01}, and a Bregman tangent majorant of $\ell$ at $z$ is associated with $h_{\majI,z}$ by Lemma \ref{lemma:bregmajorder}.
        \item[(\textbf{Maj-$\majVI$})] We apply Lemma \ref{le:majlog} to each component $(\ell_m)_{m\in \llbracket 1,M\rrbracket}$ of $\ell$ (set $\psi = \ell_m$, $\theta = H_m$, $\delta_{1}=0$, and $\delta_{2}=b_{m}$). Then, using the additivity property, a Bregman tangent majorant of $\ell$ at $z$ is associated with $h_{\majVI,z}$.       
           Furthermore, let us notice that, for every $n\in \llbracket 1,N\rrbracket$, $a_{\majVI,n}(z) \le a_{\majO,n}(z)$, 
            we have 
            \begin{equation*}
                (\forall x \in \mathcal{D})\quad  (x-z)^\top C_{h_{\majVI,z}}(x-z)\le (x-z)^\top C_{h_{\majI,z}}(x-z),
            \end{equation*}
            which yields $h_{\majVI,z} \preceq h_{\majI,z}$, showing \ref{prop:log0majorants}\ref{prop:ordre_61}.
        \end{description} 

\section{Proof of \cref{prop:quadratic_majorants}\ref{prop:ordre_35}}
\label{annex:proof35}
Let us assume that, for every $m\in \llbracket 1,M\rrbracket$, $\zeta_{m} b_{m}=\rho$.
In view of \eqref{e:majlogquad3}, \eqref{eq:endproofmajIII}, and \eqref{e:maqquad}, we have 
\begin{align*}
    (x-z)^\top & C_{h_{\majIII,z}}(x-z) \\ =& \sum_{n=1}^N a_{\majO,n}(z) \Big(\mathbbm{1}_{x_{n}\ge z_{n}} \left(\frac{1}{(z_{n}+\rho)^2}  C_{(\cdot)^2/2}(x_{n},z_{n})\right)
        + \mathbbm{1}_{x_{n}< z_{n}}\left( C_{-\ln(\cdot+\rho)}(x_{n},z_{n})\right)\Big) (x_{n}-z_{n})^2  \\
        \leq & \sum_{n=1}^N a_{\majO,n}(z) \left(\mathbbm{1}_{x_{n}\ge z_{n}} \left(\frac{1}{(z_{n}+\rho)^2}  C_{(\cdot)^2/2}(x_{n},z_{n})\right)
        + \mathbbm{1}_{x_{n}< z_{n}} \left(\frac{1}{2} c_\tau(z_n,\rho)\right) \right) (x_{n}-z_{n})^2
\end{align*}
and
\begin{equation*}
    (x-z)^\top C_{h_{\majV,z}}(x-z) = \frac12 \sum_{n=1}^N a_{\majV,n}(z)  (x_{n}-z_{n})^2 = \frac12 \sum_{n=1}^N a_{\majO,n}(z)c_\tau(z_n,\rho)  (x_{n}-z_{n})^2.
\end{equation*}

Then, to show that $h_{\majIII,z} \preceq h_{\majV,z}$, it is sufficient to prove that, for every $m\in \llbracket 1,M\rrbracket$ and $n\in \llbracket 1,N\rrbracket $ such that $x_{n}\ge z_{n}$:
\begin{equation}
    \frac{1}{(z_{n}+\rho)^2}  \le c_\tau(z_{n},\rho).
\end{equation}
By using \eqref{e:defcm}, this inequality is equivalent to
\begin{equation}
    2 \ln\Big(\frac{\rho}{\rho+z_{n}}\Big) \le -\frac{z_{n}}{z_{n}+\rho} \Big(2+\frac{z_{n}}{z_{n}+\rho}\Big).
\end{equation} 
By making the variable change
$v = (1+z_{n}/\rho)^{-1} \in (0,1]$, this can be rewritten as
\begin{equation}
2 \ln v \le (v-1)(3-v),
\end{equation}
which can easily be established by noticing that $v \mapsto  (v-1)(3-v) - 2 \ln v$ is a decaying function. Hence the result.

\section{Verifying~\cref{ass:A3:ii} for the example in \cref{sec:Appli-PET}}
\label{annexe:condition_lip}
By \cref{prop:welldef,pro:obj-cv}, for all $k \in \mathbb{N}$, $x^k$ belongs to a compact subset of $\text{dom}\, g = \mathcal{D}_0 = [\epsilon_0,+\infty)^N$. Then, it is enough to prove that, for all $k \in \mathbb{N}$, $\nabla f - \nabla h_{x^k} = \nabla \mathcal{L} + \nabla \mathcal{R} - \nabla h_{x^k}$ is Lipschitz continuous on any compact subset of $[\epsilon_0,+\infty)^N$, with a constant independent from~$k$. 

\paragraph{Lipschitz continuity of $\nabla \mathcal{L}$}
We have
\begin{equation}
    (\forall x \in \mathcal{D}) \quad \nabla \mathcal{L}(x) = \sum_{m=1}^M \left( 1 - \frac{y_m}{H_m x + b_m} \right) H_m^\top
    = H^\top \!\left( 1 - y\oslash (Hx+b) \right)
\end{equation}
where $\oslash$ denotes
the componentwise vector division.
This yields 
\begin{equation}
     (\forall x \in \mathcal{D}) \quad \nabla^2 \mathcal{L}(x) 
    = \sum_{m=1}^M \frac{y_m}{(H_m x + b_m)^2} H_m^\top H_m
    = H^\top \, \mathrm{Diag}\!\left( y\oslash (Hx+b)^{.2} \right) H,
\end{equation}
where $(\cdot)^{.2}$ is the square componentwise operation.
Since $b \in (0,+\infty)^M$, $\epsilon_0 \geq 0$, and $H \in \R^{M \times N}_{+}$, we have 
$(\forall x \in [\epsilon_0,+\infty)^N) \quad (Hx+b)_m \ge b_m > 0$
which yields 
\begin{equation}
    (\forall x \in [\epsilon_0,+\infty)^N) \quad
\spnorm{\nabla^2 \mathcal{L}(x)} 
\leq \spnorm{H}^2 \underset{1 \leq m \leq M}{\text{max}} \frac{y_m}{b_m^2} \triangleq L_\mathcal{L}. 
\end{equation}
This shows that $\mathcal{L}$ has a Lipschitz continuous gradient on $[\epsilon_0,+\infty)^N$ with a Lipschitz constant~$L_\mathcal{L}$.

\paragraph{Lipschitz continuity of $\nabla \mathcal{R}$}
Let us denote $\omega: t \mapsto \dot{\vartheta}(t)/t= (4 \delta^2)/(2 \delta^2 + t^2)^2$. Then, the gradient of the regularization term \eqref{eq:general_form_R} is expressed as
    \begin{equation}
        (\forall x \in \R^N) \quad \nabla \mathcal{R}(x) = \lambda \,  \Delta^\top  \text{Diag}\left(\Big(\omega(\| [\Delta x]_n \|) 
        {\mathbf{I}_2}
        \Big)_{1\le n \le N} \right) \Delta x + \varepsilon x.
    \end{equation}
    Moreover,
the Hessian of $\mathcal{R}$ reads
    \begin{equation}
        (\forall x \in \R^N) \;\;
    \nabla^2 \mathcal{R}(x) = \lambda \Delta^\top \text{Diag}\left(\Big(
       {\omega(\| [\Delta x]_n \|) \mathbf{I}_2 + \frac{\dot{\omega}(\| [\Delta x]_n \|)}{\| [\Delta x]_n \|} [\Delta x]_n [\Delta x]_n^\top}\Big)_{1\le n\le N}
        \right) \Delta + \varepsilon \mathbf{I}_N,
    \end{equation}
    with the quotient $\dot{\omega}(t)/t$ understood by continuous extension at $t=0$.
    Since the eigenvalues of $[\Delta x]_n [\Delta x]_n^\top$ with $n\in \llbracket 1,N\rrbracket$
    are $\|[\Delta x]_n\|^2$ and 
    $0$, and, for all $t \in \mathbb{R}$, $\dot{\omega}(t) t + \omega(t) = \ddot{\vartheta}(t)$, we have,
    for each $2\times 2$ block diagonal term, 
    \begin{equation}
    \spnorm{\omega(\| [\Delta x]_n \|) \mathbf{I}_2 + \frac{\dot{\omega}(\| [\Delta x]_n \|)}{\| [\Delta x]_n \|} [\Delta x]_n [\Delta x]_n^\top}
    =  \max\left(\omega(\| [\Delta x]_n \|),|\ddot{\vartheta}(\| [\Delta x]_n \|)|\right).
\end{equation}
    
Moreover, 
\begin{align}
    (\forall t \in \mathbb{R}) \quad \ddot{\vartheta}(t) &= \frac{4\delta^2(2\delta^2 - 3t^2)}{(2\delta^2 + t^2)^3} \in \left[-\frac{1}{4\delta^2},\frac{1}{\delta^2}\right], \quad \text{and} \quad \omega(t)\in \Big(0, \frac{1}{\delta^2}\Big].
\end{align}


Hence,
\begin{equation}
    (\forall x \in \R^N) \quad \spnorm{\nabla^2 \mathcal{R}(x)} \leq \frac{\lambda}{\delta^2} \spnorm{\Delta}^2 + \varepsilon  \triangleq L_\mathcal{R}.
    \label{eq:R_gradlip}
\end{equation}
This shows that $\mathcal{R}$ has a Lipschitz continuous gradient on $\R^N$ (and thus on $[\epsilon_0,+\infty)^N$) with Lipschitz constant $L_\mathcal{R}$.


\color{black}
\paragraph{Lipschitz continuity of $\nabla h_z$}

Let us consider the three majorant forms from Table 1, associated with their respective domain $\mathcal{D}$, and setting for $\epsilon_0$. We provide in the table below, their associated Hessian expression, and Lipschitz constants of $\nabla h_z$, for some $z\in [\epsilon_0,+\infty)^N$, i.e., $L_{h_z} > 0$ such that
\begin{equation}
    (\forall (z,x)\in\mathcal{D}_0^2) \quad \spnorm{\nabla^2 h_z(x)} \leq L_{h_z}.
\end{equation}

\begin{table}[htb!]
    \centering
    \resizebox{0.95\linewidth}{!}{%

    \begin{tabular}{|c|c|c|c|c|c|}
    \hline
         Maj. family& $\mathcal{D}$& Form of $h_z$ for $z \in \mathcal{D}$& Hessian $\nabla^2 h_z$ for $z \in \mathcal{D}$ & $\mathcal{D}_0$ & $L_{h_z}$, for $z \in \mathcal{D}_0$\\
         \hline 
         Log-shift & $(-\rho,+\infty)$ & 
         $- \sum_{n=1}^N a_n(z) \log(x_n+\rho)$& $\mathrm{Diag}\left(\left(\frac{a_n(z)}{(x_n+\rho)^2}\right)_{1\leq n \leq N}\right)$ & $[0,+\infty)$ & $\frac{1}{\rho^2}\underset{1 \leq n \leq N}{\text{max }}a_n(z)$\\
    \hline
         Log-0 & $(0,+\infty)$ & 
         $- \sum_{n=1}^N a_n(z) \log(x_n)$& $\mathrm{Diag}\left(\left(\frac{a_n(z)}{x_n^2}\right)_{1\leq n \leq N}\right)$ &$[\epsilon_0,+\infty), \epsilon_0>0$  & $\frac{1}{\epsilon_0^2}\underset{1 \leq n \leq N}{\text{max }}a_n(z)$\\
    \hline
         Quadratic & $(-\rho,+\infty)$ & 
         $\frac{1}{2}\sum_{n=1}^N a_n(z) x_n^2$& $\mathrm{Diag}\left((a_n(z))_{1\leq n \leq N}\right)$ & $[0,+\infty)$  & $\underset{1 \leq n \leq N}{\text{max}}a_n(z)$\\
    \hline
    \end{tabular}
    }
    \caption{Hessian calculations and norm majorizations for each majorant family}
    \label{tab:hessian_hz}
\end{table}

As the iterates $(x^k)_{k \in \mathbb{N}}$ lie in a compact subset of $\mathcal{D}_0$ by \cref{pro:obj-cv} and the functions $(a_n)_{1\le n\le N}$ are continuous on $\mathcal{D}_0$, there exists $L_h > 0$ such that $(\forall k \in \mathbb{N}) \; L_{h_{x^k}} \leq L_h$.

\paragraph{Conclusion} Combining the above results, for every $k\in\mathbb{N}$ and every
$(x,x')\in\mathcal{D}_0^2$, we have
\begin{equation}
\label{eq:SM3.10}
\left\|
\bigl(\nabla f-\nabla h_{x^k}\bigr)(x)
-
\bigl(\nabla f-\nabla h_{x^k}\bigr)(x')
\right\|
\leq
\bigl(L_{\mathcal L}+L_R+L_h\bigr)\|x-x'\|.
\end{equation} 
Hence, Assumption~\ref{ass:A3:ii} holds with
\[
L=L_{\mathcal L}+L_R+L_h
\] and $\dom{g}=\mathcal{D}_0$, for the three considered majorant constructions.

\vspace{0.8cm}
 \section{Detailed calculations for MM updates}
 
\label{annexe:details_updates}

\paragraph{Log-shift Majorant}
    Let us now consider the separable Log-shift majorant \textbf{Maj-\majO} \eqref{eq:h0} for $\mathcal{L}$. We have, for all $z\in\mathcal{D}$,
\begin{align*}
   (\forall x\in\mathcal{D})\quad D_{h_{\majO,z}}(x,z) = -\sum_{n=1}^N a_{\majO,n}(z)\ln\left(\frac{x_n+\rho}{z_n+\rho}\right) - \frac{a_{\majO,n}(z)}{z_n+\rho}(x_n-z_n).
\end{align*}

Then, for every $n\in \llbracket 1,N\rrbracket$,
\begin{align*}
    (\forall x\in\mathcal{D})\quad  [\nabla_x Q_f(x,z) ]_n &= [\nabla f(z)]_n - \frac{a_{\majO,n}(z)}{x_n +\rho}+ \frac{a_{\majO,n}(z)}{z_n+\rho} + M_\mathcal{R} (x_n - z_n) \\
    &= d_n(z) - \frac{a_{\majO,n}(z)}{x_n +\rho} + M_\mathcal{R} x_n
\end{align*}
where $d_n(z) = [\nabla f(z)]_n + \frac{a_{\majO,n}(z)}{z_n+\rho} - M_\mathcal{R} z_n$. Then
\begin{align*}
     [\nabla_x Q_f(x,z) ]_n = 0 &\;\;\Leftrightarrow\;\;  M_\mathcal{R} x_n^2 + (d_n(z) + M_\mathcal{R} \rho) x_n+\rho d_n(z) - a_{\majO,n}(z) = 0, 
\end{align*}

 which yields the following Log-shift MM iteration:
 
\begin{equation}
    \boxed{(\forall k \in \N^\ast) \quad x^{k+1} = \operatorname{proj}_{\mathcal{D}_0}\left(\Big(\frac{\sqrt{(d_n(x^k)-M_\mathcal{R} \rho)^2+4M_\mathcal{R} a_{\majO,n}(x^k)}-d_n(x^k)-M_\mathcal{R} \rho}{2M_\mathcal{R} }\Big)_{1\leqslant n \leqslant N}\right).}
\end{equation}
The same goes for majorant \textbf{Maj-\majVIII} which has a similar structure.

\paragraph{Log-0 Majorant}
Let us now consider the separable Log-0 majorant \textbf{Maj-}$\majI$ for $\mathcal{L}$ \eqref{eq:h1}. 
We have, for all $z \in\mathcal{D}$,
\begin{equation*}
   (\forall x\in\mathcal{D}) \quad D_{h_{\majI,z}}(x,z) = -\sum_{n=1}^N a_{\majO,n}(z)\ln\left(\frac{x_n}{z_n}\right) - \frac{a_{\majO,n}(z)}{z_n}(x_n-z_n).
\end{equation*}

Then similarly to the previous majorant family, we have
\begin{equation*}
    [\nabla_x Q_f(x,z) ]_n = d_n(z) - \frac{a_{\majO,n}(z)}{x_n} + M_\mathcal{R} x_n
\end{equation*}
where $d_n(z) = [\nabla f(z)]_n + \frac{a_{\majO,n}(z)}{z_n} - M_\mathcal{R} z_n$. Then
\begin{align*}
     [\nabla_x Q_f(x,z) ]_n = 0 &\;\;\Leftrightarrow\;\;  M_\mathcal{R} x_n^2 + d_n(z) x_n - a_{\majO,n}(z) = 0, 
\end{align*}
which yields the following Log-0 MM iteration:
\begin{equation}
    \boxed{(\forall k \in \N^\ast) \quad x^{k+1} = \operatorname{proj}_{\mathcal{D}_0}\left(\Big(\frac{\sqrt{(d_n(x^k))^2+4M_\mathcal{R} a_{\majO,n}(x^k)}-d_n(x^k)}{2M_\mathcal{R} }\Big)_{1\leqslant n \leqslant N}\right).}
\end{equation}
A similar expression holds for majorant \textbf{Maj}-\majVI.
 
\paragraph{Quadratic majorant}
Let us finally consider the quadratic majorant \textbf{Maj-\majVp} \eqref{eq:h5p} for $\mathcal{L}$. The Bregman tangent majorant of $f$ at $z\in\mathcal{D}$ then reads: 

\begin{align*}
(\forall x\in\mathcal{D})\quad 
    Q_f(x,z) &= f(z) + \langle \nabla f(z),x-z\rangle + D_{h_{\majVp,z}}(x,z)+ \frac{M_\mathcal{R}}{2}\Vert x-z \Vert^2 &\text{ \eqref{eq:Qf_withdiv}} \\
    &= f(z) + \langle \nabla f(z),x-z\rangle +
    \frac12\sum_{n=1}^{N}a_{\majVp,n}(z)(x_n-z_n)^2+ \frac{M_\mathcal{R}}{2}\Vert x-z \Vert^2. &
\end{align*}
Then, we have 
\begin{equation}
     (\forall x\in\mathcal{D})\quad  \nabla_x Q_f(x,z) = \text{Diag}\big((a_{\majVp,n}(z)+M_\mathcal{R})_{1\le n \le N}\big)(x-z) + \nabla f(z),
\end{equation}
which yields the following quadratic majorant MM iteration \eqref{eq:VBMMM_update}:
\begin{equation}
\boxed{(\forall k \in \N^\ast) \quad x^{k+1} = \operatorname{proj}_{\mathcal{D}_0}\left(\Big(x_n^k - \frac{[\nabla f(x^k)]_n}{a_{\majVp,n}(x^k)+M_\mathcal{R}}\Big)_{1\leqslant n \leqslant N}\right).}
\end{equation}

\clearpage
\section{Synthetic view of majorant order relations} \textcolor{white}{.} 
\begin{figure}[htb!]
    \centering
\resizebox{0.7\textwidth}{!}{%
\begin{tikzpicture}[
  node distance = 1cm and 2cm,
  box/.style = {
    draw, rectangle,
    minimum width=3.2cm,
    minimum height=1cm,
    align=center,
    font=\small,
    fill=white,
    draw=black
  },
  arrow/.style = {-, black},
  every node/.style = {font=\small}
]

\node[box] (maj8) {\textbf{Maj-\majVIII} $\mu = \rho$ (Log-shift)};
\node[box, below=0.5cm of maj8] (maj0) {\textbf{Maj-\majO} (Log-shift)};
\node[box, below left=0.5cm and 1.5cm of maj0]  (maj1) {\textbf{Maj-\majI} (Log-0)};
\node[box, below=0.5cm of maj0]                  (maj2) {\textbf{Maj-\majII} (Log-shift)};
\node[box, below right=0.5cm and 1.5cm of maj0]  (maj3) {\textbf{Maj-\majIII} (Log-shift)};

\node[box] at (maj1 |- maj8) (maj6) {\textbf{Maj-\majVI} (Log-0)};

\node[box, below=1.5cm of maj3] (maj5)  {\textbf{Maj-\majV} (quadratic)};
\node[box, below=0.5cm of maj5] (maj5p) {\textbf{Maj-\majVp} (quadratic)};

\node[above=0.55cm of maj5, font=\small\itshape] (cond)
  {if $\zeta_m b_m = \rho \ \forall m$};

\draw[arrow] (maj8)     -- node[at start, rotate=-90, anchor=west, font=\small, black, fill=white, inner sep=1pt] {$\preceq$} (maj0);
\draw[arrow] (maj0)     -- node[at start, rotate=-90, anchor=west, font=\small, black, fill=white, inner sep=1pt] {$\preceq$} (maj2);
\draw[arrow] (maj6)     -- node[at start, rotate=-90, anchor=west, font=\small, black, fill=white, inner sep=1pt] {$\preceq$} (maj1);
\draw[arrow] (maj3)     -- node[at start, rotate=-90, anchor=west, font=\small, black, fill=white, inner sep=1pt] {} (cond.north);
\draw[arrow] (cond)     -- node[at start, rotate=-90, anchor=west, font=\small, black, fill=white, inner sep=1pt] {$\preceq$} (maj5);
\draw[arrow] (maj5)     -- node[at start, rotate=-90, anchor=west, font=\small, black, fill=white, inner sep=1pt] {$\preceq$} (maj5p);

\draw[arrow] (maj0) -- node[pos=0.05, rotate=-165, anchor=west, font=\small, black, fill=white, inner sep=1pt] {$\preceq$} (maj1);
\draw[arrow] (maj0) -- node[pos=0.24, rotate=-27,  anchor=east, font=\small, black, fill=white, inner sep=1pt] {$\preceq$} (maj3);
\end{tikzpicture}%
}

    \caption{Diagram representing all order relations (\cref{def:btm_order}) established between majorants in \cref{subsec:family-maj} (tightest ones on top).}
    \label{fig:orderrelations}
\end{figure}

\section{Number of projector calls for each majorant update} \textcolor{white}{.} 
\label{annex:nbproj}
\begin{table}[htb!]
    \centering
\footnotesize{
    \begin{tabular}{|c|c|c|c|}
    \hline
         Family & Majorant & Projections & Backprojections \\
    \hline
         \multirow{4}{*}{Log-shift} & Maj-\majO   & \multirow{4}{*}{1} & 2 \\
    \cline{2-2} \cline{4-4}
         & Maj-\majII $\;$(BPG) & & 1(*) \\
    \cline{2-2} \cline{4-4}
         & Maj-\majIII & & 2 \\
    \cline{2-2} \cline{4-4}
         & Maj-\majVIII & & 1 \\
    \hline
         \multirow{2}{*}{Log-0} & Maj-\majI  & \multirow{2}{*}{1} & 2 \\
    \cline{2-2} \cline{4-4}
         & Maj-\majVI & & 1 \\
    \hline
         \multirow{4}{*}{Quadratic} & Maj-\majV   & \multirow{4}{*}{1} & $M$ \\
    \cline{2-2} \cline{4-4}
         & Maj-\majVp $\;$(VMFB)  & & 2 \\
    \cline{2-2} \cline{4-4}
         & Maj-\majVII & & 2 \\
    \cline{2-2} \cline{4-4}
         & Maj-Lip (proj. gradient) & & 1 \\
    \hline
     \end{tabular}
}
    \caption{Number of projection/backprojection operations per iteration for each choice of majorant. \reviewchanges{(*) Before iterating, \textbf{Maj-}$\majII$ requires a precomputing step of $M^2$ backprojections (which in practice is done once for each phantom slice and then reused).}
    }
    \label{tab:operations}
\end{table}
\section{Details about the analytical simulation process}
\label{annex:simulation}

\reviewchanges{
In this proof-of-concept study, we used simulated data obtained from a real $^{18}$F-FDG high resolution brain PET exam (healthy volunteer) recorded with a HRRT PET system (Siemens Healthineers, Forchheim, Germany). After segmenting the reconstructed PET exam into 100 anatomical regions of interest using the T1-weighted MRI images, we built a piecewise-constant phantom from which we extracted two 2D slices. These central slices were chosen in order to include cortical and subcortical structures (basal ganglia). We then carried out analytical simulations to build synthetic sinograms (of size 336$\times$336) from this phantom as imaged by a Biograph TruePoint TrueV system (Siemens Healthineers, Forchheim, Germany) following the methodology described in \cite{analytical_simulator}, with the following direct model:
$$y_m=\mathcal{P}\Big(\underbrace{C_m^\mathrm{sens} C_m^\mathrm{att} X_m G}_{H_m}\bar{x}+\underbrace{s_m+r_m}_{b_m}\Big)$$
where, for line of response $m$,
\begin{itemize}
    \item to account for intrinsic physical resolution effects, the activity image $\bar{x}$ is first convolved with an isotropic two-dimensional Gaussian kernel (represented by a convolution matrix $G$) modeling the scanner point spread function in image space, with full width at half maximum (FWHM) of 4\,mm;
    \item $X_m$ is the corresponding geometrical projection operator, here implemented using Joseph's approach (Joseph, P.M. (1982). \textit{An Improved Algorithm for Reprojecting Rays through Pixel Images.} IEEE Transactions on Medical Imaging, 1, 192-196.), computed on the fly without assuming specific geometric structure of the camera;
    \item $C_m^\mathrm{sens}$ and $C_m^\mathrm{att}$ are the corresponding sensitivity and attenuation factors;
    \item $s_m$ and $r_m$ represent respectively the expectation of scattered and random coincidences in the line of response $m$ (strictly positive additive background);
    \item $\mathcal{P}$ is the Poisson noise operator;
    \item $y_m$ is the resulting measured data; for each slice, about 7 millions of counts are measured.
\end{itemize}
}
\vspace{0.5cm}

\section{Details about image quality metrics}
\label{annex:metrics}
\reviewchanges{We provide details about quality metrics calculations in \cref{tab:table_15s,tab:table_120s,tab:kkt_0p001} and \cref{fig:exreconstructions_slice65,fig:exreconstructions_slice71}. Herebelow, $x$ denotes the estimated image after the given stopping criterion, and $\bar{x}$ the ground-truth image, both in $[0,+\infty)^N$.
\begin{itemize}
    \item NRMSE (Normalized Root Mean Square Error): 
        $$\text{NRMSE}(x) = \frac{\sqrt{\frac{1}{N}{\sum_{n=1}^N (x_n-\bar{x}_n)^2}}}{\text{mean}(\bar{x})};$$
    \item SSIM (Structural Similarity Index Measure, Wang, Z. \textit{et al.} (April 2004), \textit{Image quality assessment: from error visibility to structural similarity}, in IEEE Transactions on Image Processing, vol. 13, no. 4, pp. 600-612): calculated using \texttt{scikit-image} with default parameters and a range covering the minimum and maximum values of the image.
    \item PSNR (Peak Signal to Noise Ratio): 
        $$\text{PSNR}(x) = 20 \log_{10}\left(\frac{\max(\bar{x})}{\sqrt{\frac{1}{N}\sum_{n=1}^N (x_n-\bar{x}_n)^2}}\right);$$
    \item CNR (Contrast to Noise Ratio):
        $$\text{CNR}(x) = \frac{\text{mean}((x_n)_{n\in\mathcal{N}_{\text{hot}}})-\text{mean}((x_n)_{n\in\mathcal{N}_{\text{cold}}})}{\sqrt{\text{std}((x_n)_{n\in\mathcal{N}_{\text{hot}}})^2+\text{std}((x_n)_{n\in\mathcal{N}_{\text{cold}}})^2}},$$
        where $\mathcal{N}_{\text{hot}}$ and $\mathcal{N}_{\text{cold}}$ designate two regions of interest where activity is respectively high and low for the tracer considered (in our case, the putamen or caudate nucleus as ``hot'' region, and the white matter as ``cold'' region).
\end{itemize}}

\makeatletter
\let\addcontentsline\oldaddcontentsline
\makeatother
\end{appendices}

\end{document}